\documentclass{siamltex}
\usepackage{amsmath}
\usepackage{amssymb}
\usepackage{graphicx}
\usepackage{enumerate}
\usepackage{color}

\def\ep{\epsilon}
\def\r{{\mathbf r}}
\def\f{\frac}
\def\om{\omega}
\def\Om{\Omega}
\def\na{\nabla}
\def\q{\quad}
\def\etal{{\it et al}\;}
\def\eref#1{(\ref{#1})}
\def\p{\partial}

\def\H{{\mathbf H}}

\def\a{{\mathbf a}}
\def\e{{\mathbf e}}

\def\n{{\mathbf n}}

\title{Magnetic resonance-based reconstruction method of conductivity and permittivity distributions at the Larmor frequency\thanks{\footnotesize This work was supported  by the ERC Advanced Grant Project MULTIMOD--267184  and the National Research Foundation of Korea (NRF) grant funded by the Korean government (MEST) (No. 2011-0028868, 2012R1A2A1A03670512).}}
\author{Habib Ammari\thanks{\footnotesize Department of Mathematics and Applications,
Ecole Normale Sup\'erieure, 45 Rue d'Ulm, 75005 Paris, France
(habib.ammari@ens.fr).} \and Hyeuknam Kwon\thanks{\footnotesize Department of Computational
Science and Engineering, Yonsei University, 50 Yonsei-Ro,
Seodaemun-Gu, Seoul 120-749, Korea (3c273-85@hanmail.net, hansubman@naver.com, seoj@yonsei.ac.kr).} \and Yoonseop Lee\footnotemark[3]
\and Kyungkeun Kang\footnotemark[4]\thanks{\footnotesize Department of mathematics
Yonsei University 50 Yonsei-Ro, Seodaemun-Gu, Seoul 120-749, Korea (kkang@yonsei.ac.kr).} \and Jin Keun Seo\footnotemark[3]}

\begin{document}

\maketitle

\begin{abstract}
Magnetic resonance electrical property tomography  is a recent medical imaging modality for visualizing the electrical tissue properties  of the human body using  radio-frequency magnetic fields. It uses the fact that in magnetic resonance imaging systems the eddy currents induced by the radio-frequency magnetic fields reflect the conductivity ($\sigma$)  and permittivity ($\ep$) distributions inside the tissues through Maxwell's equations. The corresponding inverse problem consists of reconstructing the admittivity distribution ($\gamma=\sigma+i\om\ep$) at the Larmor frequency ($\om/2\pi=$128 MHz for a 3 tesla MRI machine) from the positive circularly polarized component of the magnetic field $\H=(H_x,H_y,H_z)$. Previous methods are usually based on an assumption of local homogeneity ($\na\gamma\approx 0$) which simplifies the governing equation. However, previous methods that include  the assumption of homogeneity  are prone to artifacts in the region where $\gamma$ varies.  Hence, recent work has sought  a reconstruction method that does not assume local-homogeneity. This paper presents a new magnetic resonance electrical property tomography reconstruction method which does not require any local homogeneity assumption on $\gamma$. We find that $\gamma$ is a solution of a semi-elliptic partial differential equation with its coefficients depending only on the measured data $H^+$, which enable us to compute a blurred version of $\gamma$. To improve the resolution of the reconstructed image, we developed a new optimization algorithm that minimizes the mismatch between the data and the model data as a highly nonlinear function of $\gamma$. Numerical simulations are presented to illustrate the potential of the proposed reconstruction method.
\end{abstract}



\markboth{H. Ammari et al.}{MR-based reconstruction method of conductivity and permittivity distributions}

\section{Introduction}
Magnetic resonance imaging (MRI) system can visualize both the conductivity, $\sigma$, and permittivity, $\epsilon$, of biological tissues at the  Larmor frequency, which is approximately 128 MHz for a 3 tesla MRI machine. Magnetic resonance electrical property tomography (MREPT) uses a time-harmonic magnetic field inside an imaging object.  The standard radio-frequency coil of the magnetic resonance scanner produces the field by feeding in a sinusoidal current at the Larmor frequency. The time-harmonic magnetic field, denoted by $\H=(H_x,H_y,H_z)$, reflects both the conductivity $\sigma$ and permittivity $\epsilon$ of human tissues through the following arrangement of time-harmonic Maxwell's equations:
\begin{equation}\label{eq:governH_0}
  -\Delta \H = \na\log (\sigma+i\om\ep)\times[\na\times\H] - i\omega\mu_0(\sigma+i\om\ep)\H\q\q \mbox{in }~ \Om,
\end{equation}
where $\mu_0=4\pi\times 10^{-7}$  H$/$m is the magnetic permeability of free space,  $\om/2\pi$ is the Larmor frequency of the MRI scanner, and $\Om$ denotes a three dimensional domain occupying an imaging object. Here, we use the fact that the magnetic permeability of the human body is approximately equal to $\mu_0$.

Clinical MRI scanners  measure  the positive rotating magnetic field, $H^+:=(H_x+ i H_y)/2$, which is the component of the magnetic field $\H$ in the direction $(1,i,0)/2$. This is because the MR signal, denoted by $S$, contains  partial information about the time-harmonic magnetic field $\H=(H_x,H_y,H_z)$ in the following way
\begin{equation}\label{MRsignal}
  S_\tau(\r)~~ \propto ~~M(\r) H^-(\r) H^+(\r) \f{\sin (\alpha\tau|H^+(\r)|)}{|H^+(\r)|}\q\mbox{ for }~\r=(x,y,z)\in \Om,
\end{equation}
where   $H^-=(H_x- i H_y)/2$  is the negative rotating magnetic field,  $M(\r)$ is the standard MR magnitude image at position $\r$, and $\alpha$ is a constant. Here, $\tau$ is the duration of the radio-frequency pulse that controls the intensity of the signal $S_\tau$. Acquiring two MR signals $S_{\tau_1}$ and $S_{\tau_2}$ with suitably chosen $\tau_1$ and $\tau_2$, we  can extract the  $H^+$  data through \eref{MRsignal} with the assumption that ${H^+}/{|H^+|}\approx{H^-}/{|H^-|}$.  This data acquisition technique is called B1 mapping, and  was first suggested by Haacke \etal. \cite{Haacke1991} in the early nineties. For details on the B1 mapping technique measuring $H^+$, we refer to numerous published works in the literature \cite{Akoka1993,Collins2002,Moortele2005,Stollberger1996,Wang2005}.

The inverse problem of MREPT consists of reconstructing distributions of $\sigma$ and $\ep$ from $H^+$. To solve the inverse problem, we need to represent the distributions of $\sigma$ and $\ep$ with respect to the data $H^+$. Under the assumption of the local homogeneity, $\na (\sigma+i\om\ep)={\bf 0}$,
the governing partial differential equation \eref{eq:governH_0} directly gives the following simple relation between $\sigma+i\om\ep$ and $H^+$;
\begin{equation}\label{eq:governH_1}
  -\Delta H^+= - i\omega\mu_0(\sigma+i\om\ep) H^+\q\q \mbox{in }~ \Om.
\end{equation}
The most widely used MREPT reconstruction methods \cite{Wen2003,Katscher2007,Katscher2006a,Katscher2006b,Katscher2009}  are based on \eref{eq:governH_1} as it gives the direct representation formula for $\sigma+i\om\ep$ with respect to $H^+$,
\begin{equation}\label{eq:direct-formula}
  \sigma+i\om\ep = \f{1}{i\om\mu_0} \f{\Delta H^+}{H^+}\q\q \mbox{in }~ \Om.
\end{equation}
However, when $\na (\sigma+i\om\ep)$ is not small, the direct formula \eref{eq:direct-formula} produces serious reconstruction errors \cite{Seo2012a}. The local homogeneity assumption neglects the contribution of
$ \na\ln\gamma\times(\na\times\H)$  in \eref{eq:governH_0}. Such reconstruction errors are rigorously analyzed in \cite{Seo2012a}.

We need to remove the local homogeneity assumption to develop a reconstruction method. Recently, a reconstruction method \cite{Song2013} removing the assumption of $(\frac{\partial}{\partial x},\frac{\partial}{\partial y})(\sigma+i\om\ep)=\mathbf{0}$ has been developed, although it  still requires the assumption of  $\f{\p}{\p z}(\sigma+i\om\ep)=0$. The method is based on the finding that, under the assumption of longitudinal homogeneity, $\sigma+i\om\ep$ is a solution of a semilinear elliptic PDE with coefficients that only depend on $H^+$ \cite{Song2013}.

In this paper, with no assumption of local homogeneity for $\sigma+i\om\ep$, we develop a new reconstruction method. We find that $\sigma$ and $\ep$ satisfy the elliptic partial differential equation,
\begin{equation}\label{ellipticPDE}
  \na\cdot\left(G_2[H^+]\na\left(\begin{array}{c}
    \sigma \\
    \ep
  \end{array}\right)\right) + G_1[\sigma,\epsilon,H^+]\cdot\na\left(\begin{array}{c}
    \sigma \\
    \ep
  \end{array}\right)
 +G_0[\sigma,\ep,H^+]
 =0 \q \mbox{in }~ \Om,
\end{equation}
where $G_2[H^+]$ is a positive semi-definite matrix, and $G_1[\sigma,\epsilon,H^+]$ and $G_0[\sigma,\ep,H^+]$ are vector fields depending only on $\sigma$, $\ep$, and $H^+$.
Hence, the distribution of $\sigma$ and $\ep$ can be obtained by solving  equation (\ref{ellipticPDE}). Unfortunately, $G_2[H^+]$ in \eref{ellipticPDE} is degenerate, and thus requires the addition of the weighted diffusion, $\rho$, to the $(3,3)$ entry of $G_2[H^+]$ so that $G_2[H^+]+\rho\e_3^T\e_3$ is positive definite, where $\e_3=(0,0,1)$. Thus,  \eref{ellipticPDE} yields  blurred images of $\sigma$ and $\ep$ in the $z$-direction.

To improve the spatial resolution of the reconstructed image, we develop an  optimal control method for the parameters $\sigma$ and $\ep$. In the proposed  adjoint-based optimization method, the conductivity and permittivity  distributions are updated iteratively by a nonlinear optimization algorithm which minimizes the discrepancy function describing the $L^2$-mismatch between the forward model and the observed data.
We compute the Fr\'echet derivatives of the discrepancy function with respect to $\sigma$ and $\ep$. This optimal control method requires a very good initial guess. Fortunately, we can obtain a good initial guess using \eref{ellipticPDE}. Several numerical simulations are carried out to show the validity of the  proposed reconstruction method.

\section{Governing equation for the admittivity reconstruction}\label{govern}
We assume that an imaging object occupying a three-dimensional domain $\Om$ with its boundary $\p\Om$ being of class  $\mathcal C^2$. Let $\gamma=\sigma+i\om\ep$ denote the  admittivity of the subject at the MR Larmor frequency. For simplicity, we assume that
$\gamma$ is a constant near the boundary; that is, $\gamma=\gamma_0$ in the region $\Omega_d:=\{x\in\Omega ~|~ \mathrm{dist}(x,\partial\Omega)<d\}$ for some $d>0$, where $\gamma_0=\sigma_0 + i \omega \ep_0$ with $\sigma_0$ and $\ep_0$ being known reference quantities.

Let $H^s(\Omega)$ denote the standard Sobolev space of order $s$. We assume that the admittivity distribution $\gamma=\sigma+i\om\epsilon$  belongs to the following admissible set $\mathcal A$:
\begin{equation} \label{defA}
  \mathcal A = \left\{\gamma\in H^2(\Omega)\cap L^\infty_{\underline{\lambda}, \overline{\lambda}}(\Om) ~|~ \omega \mu_0  \|\gamma\|_{H^2}+ 8 |\Omega|^{1/6} \|\f{\nabla\gamma}{\gamma}\|_{H^2}<c_1,~ \left.\gamma\right|_{\Omega_d}=\gamma_0\right\},
  \end{equation}
where $\underline{\lambda}$, $\overline{\lambda}$ and $c_1$ are positive constants, $|\Omega|$ denotes the volume of $\Omega$,  and $$L^\infty_{\underline{\lambda}, \overline{\lambda}}(\Om):= \bigg\{ \gamma\in L^\infty(\Om):~
\underline{\lambda}<\Re\{\gamma\}, \Im\{\gamma\}<\overline{\lambda} \bigg\}.$$

The inverse problem is to invert the map $\gamma \to H^+$ where $H^+$ represents the measured data extracted from the MR signal in \eref{MRsignal} and the relation between $\H$ and $\gamma$ is given in  \eref{eq:governH_0}.
Noting that the component $H_z$ is known to be relatively small with a regular birdcage coil of  MRI scanner  \cite{Vaughan2004},  we  assume $H_z=0$.

To solve the inverse problem, we need to express $\sigma+i\om\ep$  in terms of $H^+$ only using the governing equation \eref{eq:governH_0}.
Taking the inner product of both sides of equation \eref{eq:governH_0} with the vector $\a=(1,i,0)/2$, we have
\begin{equation}\label{BB1}
  -\Delta H^+ = (\nabla\log\gamma\times(\nabla\times\H))\cdot\a-i\omega\mu_0\gamma H^+ \q\mbox{in}~\Om.
\end{equation}
It follows from the result of \cite{Song2013} that the contribution of $H^-$ in \eref{BB1} can be eliminated from the identity
\begin{eqnarray*}
 &&(\nabla\log\gamma\times(\nabla\times\H))\cdot\a\\
  &&\q\q\q=
  -\nabla\log\gamma\cdot\left(\frac{\partial H^+}{\partial x}-i\frac{\partial H^+}{\partial y}, i\frac{\partial H^+}{\partial x}+\frac{\partial H^+}{\partial y}, \frac{\partial H^+}{\partial z}\right).
\end{eqnarray*}
Equation \eref{BB1} with the above identity gives the following lemma \cite{Song2013}.
\begin{lemma}\label{lemma:H+}
The $\gamma$ in \eref{eq:governH_0} satisfies the following first-order partial differential equation
\begin{equation}\label{eq:H+}
 \mathcal L H^+
  \cdot
  \f{\na \gamma}{ \gamma} - i\omega\mu_0 \gamma ~H^+~~= ~-\Delta H^+\q\mbox{in}~\Om,
\end{equation}
where $\mathcal L$ is the linear differential operator given by
\begin{equation}\label{eq:operatorL}
  \mathcal L=\left(-\f{\p }{\p x} + i\f{\p }{\p y},~~-i\f{\p }{\p x} - \f{\p }{\p y},~~ -\f{\p}{\p z}\right).
\end{equation}
\end{lemma}

According to Lemma \ref{lemma:H+}, the inverse problem is reduced to solve the first-order partial differential equation \eref{eq:H+} for $\gamma$. Unfortunately, it may not be possible to solve the first-order partial differential equation \eref{eq:H+}. As the direction vector field of $\mathcal L H^+$ is not a real-valued function, the method of characteristics can not be applied.
Indeed, H\"ormander \cite{Holmander1966} and Lewy \cite{Lewy1957} provided non-existence results for the first order partial differential equation with complex-valued coefficients. To be precise, the governing equation \eref{eq:H+} for $H^+$ can be rewritten in the standard form $F\cdot\na u = f(\cdot,u)$, where $F=\mathcal L H^+$, $u=\log\gamma$, and $f(\cdot,u)=i\omega\mu_0e^{u}H^+-\Delta H^+$.
According to the Cauchy-Kowalevski theorem \cite{Kowalevsky1875}, the equation $F\cdot\na u = f(\cdot,u)$ with  suitable initial data, can be locally solvable only when $f$ is analytic. On the other hand, this local solvability can not be guaranteed for general $f\in \mathcal C^\infty$ from Lewy's example \cite{Lewy1957}. This is why we do not use the model \eref{eq:H+} to compute $\gamma$.


\subsection{Elliptic equation for the admittivity}\label{method1}
In this subsection, we prove that $\sigma$ and $\epsilon$ satisfy the elliptic partial differential equation \eref{ellipticPDE} which is one of our main results in this paper. This key observation follows from long and careful computations.

\begin{theorem}\label{thm:ellipticPDE}
The distributions of $\sigma$ and $\ep$ satisfy the following equation:
\begin{equation}\label{eq:poisson-sigma-epsilon}
  \na\cdot\left(A[H^+]\na\left(\begin{array}{c}
    \sigma \\
    \om\ep
  \end{array}\right)\right)
  +
  F_0[H^+]\cdot\na\left(\begin{array}{c}
    \sigma \\
    \om\ep
  \end{array}\right)=\left(\begin{array}{c}
    F_1[\sigma,\ep,H^+] \\
    F_2[\sigma,\ep,H^+]
  \end{array}\right)\q\mbox{in }~\Om,
\end{equation}
where $A[H^+]$ is a positive semi-definite matrix given by
\begin{equation}\label{eq:postiveMatix}
\hskip -0.6in A[H^+] =
    \left[\begin{array}{ccc}
      P_x^2+P_y^2 & 0 & P_xP_z+P_yQ_z \\
      0 & P_x^2+P_y^2 & P_yP_z-P_xQ_z \\
      P_xP_z+P_yQ_z & P_yP_z-P_xQ_z & P_z^2+Q_z^2
    \end{array}\right] \q\mbox{in }~\Om .
\end{equation}
Here, $F_0[H^+]$, $F_1[\sigma,\ep,H^+]$, and $F_2[\sigma,\ep,H^+]$ are given by
\begin{equation}\label{eq:matrixF0}
  F_0=
    -\left[\begin{array}{rcll}
      P_x[H^+]\na\cdot P[H^+] + Q_x[H^+]\na\cdot Q[H^+] \\
      P_y[H^+]\na\cdot P[H^+] + Q_y[H^+]\na\cdot Q[H^+] \\
      P_z[H^+]\na\cdot P[H^+] + Q_z[H^+]\na\cdot Q[H^+]
    \end{array}\right],
\end{equation}
\begin{eqnarray}
  F_1 &=& -P[H^+]\cdot\na\phi[\sigma, \ep,H^+]+Q[H^+]\cdot\na\psi[\sigma, \ep,H^+]+E[\om\ep,H^+], \label{eq:defF1} \\
  F_2 &=& -Q[H^+]\cdot\na\phi[\sigma, \ep,H^+]-P[H^+]\cdot\na\psi[\sigma, \ep,H^+]-E[\sigma,H^+], \label{eq:defF2}
\end{eqnarray}
where $P[H^+]$, $Q[H^+], E[\eta,H^+]$, $\phi[\sigma, \ep,H^+]$ and $\psi[\sigma, \ep, H^+]$ are defined by
\begin{eqnarray}
  && P=(P_x,P_y,P_z)= \left(-\f{\p}{\p x}H^+_r-\f{\p}{\p y}H^+_i,~ \f{\p}{\p x}H^+_i-\f{\p}{\p y}H^+_r,~ -\f{\p}{\p z}H^+_r\right), \label{eq:defP} \\
  && Q=(Q_x,Q_y,Q_z)= \left(\f{\p}{\p x}H^+_i - \f{\p}{\p y}H^+_r,~ \f{\p}{\p x}H^+_r + \f{\p}{\p y}H^+_i,~ \f{\p}{\p z}H^+_i\right), \label{eq:defQ} \\
  && E[\eta,H^+]=Q[H^+]\cdot\na(P[H^+]\cdot\na \eta)-P[H^+]\cdot\na(Q[H^+]\cdot\na \eta), \label{eq:defE} \\
  && \phi= \om\mu_0H^+_i\sigma^2 - \om^3\mu_0H^+_i\ep^2 + 2\om^2\mu_0H^+_r\sigma\ep + \Delta H^+_r\sigma -\om\Delta H^+_i\ep, \label{eq:defphi} \\
  && \psi=-\om\mu_0H^+_r\sigma^2 + \om^3\mu_0H^+_r\ep^2 + 2\om^2\mu_0H^+_i\sigma\ep + \Delta H^+_i\sigma +\om\Delta H^+_r\ep. \label{eq:defpsi}
\end{eqnarray}
\end{theorem}
\begin{proof}
We first separate the governing equation \eref{eq:H+} into its real and imaginary parts. Let $H^+_r$ and $H^+_i$ be the real and imaginary parts of $H^+$, i.e., $H^+ = H^+_r + iH^+_i$.
Let $\gamma_i$ be the imaginary part of the admittivity.

Equation \eref{eq:H+} can be expressed as
\begin{equation}\label{eq:Govern-2}
  -\Delta H^+ \gamma
  =
  \left(-\f{\p H^+}{\p x} + i\f{\p H^+}{\p y}\right)\f{\p\gamma}{\p x}
  -\left(i\f{\p H^+}{\p x} + \f{\p H^+}{\p y}\right)\f{\p\gamma}{\p y}
  -\f{\p H^+}{\p z}\f{\p\gamma}{\p z}
  -i\om\mu_0H^+\gamma^2 .
\end{equation}
The real and imaginary parts of equation \eref{eq:Govern-2} are given, respectively, by
\begin{eqnarray}
  \Delta H^+_r\sigma-\Delta H^+_i\gamma_i
  &=&
  \left(\f{\p H^+_r}{\p x}+\f{\p H^+_i}{\p y}\right)\f{\p\sigma}{\p x}
  -\left(\f{\p H^+_i}{\p x}-\f{\p H^+_r}{\p y}\right)\f{\p\gamma_i}{\p x} \nonumber
  \\
  &~& \hspace{-1cm}
  +\left(\f{\p H^+_r}{\p y}-\f{\p H^+_i}{\p x}\right)\f{\p\sigma}{\p y}
  -\left(\f{\p H^+_i}{\p y}+\f{\p H^+_r}{\p x}\right)\f{\p\gamma_i}{\p y} \nonumber
  \\
  &~& \hspace{-1cm}
  +\f{\p H^+_r}{\p z}\f{\p\sigma}{\p z}
  -\f{\p H^+_i}{\p z}\f{\p\gamma_i}{\p z}
  -\om\mu_0\left(2H^+_r\sigma\gamma_i+H^+_i(\sigma^2-\gamma_i^2)\right) \label{eq:realGovern-2}
\end{eqnarray}
and
\begin{eqnarray}
  \Delta H^+_r\gamma_i+\Delta H^+_i\sigma
  &=&
  \left(\f{\p H^+_r}{\p x}+\f{\p H^+_i}{\p y}\right)\f{\p\gamma_i}{\p x}
  +\left(\f{\p H^+_i}{\p x}-\f{\p H^+_r}{\p y}\right)\f{\p\sigma}{\p x} \nonumber
  \\
  &~& \hspace{-1cm}
  +\left(\f{\p H^+_r}{\p y}-\f{\p H^+_i}{\p x}\right)\f{\p\gamma_i}{\p y}
  +\left(\f{\p H^+_i}{\p y}+\f{\p H^+_r}{\p x}\right)\f{\p\sigma}{\p y} \nonumber
  \\
  &~& \hspace{-1cm}
  +\f{\p H^+_r}{\p z}\f{\p\gamma_i}{\p z}
  +\f{\p H^+_i}{\p z}\f{\p\sigma}{\p z}
  +\om\mu_0\left(H^+_r(\sigma^2-\gamma_i^2)-2H^+_i\sigma\gamma_i\right). \label{eq:imagGovern-2}
\end{eqnarray}
The real part \eref{eq:realGovern-2} can be written as
\begin{equation}
    P[H^+]\cdot\na\sigma + Q[H^+]\cdot\na\gamma_i + \kappa[\gamma,H^+] = 0 \quad\mbox{in}~\Omega,
    \label{eq:realGovern-3}
\end{equation}
where $\kappa:=-\om\mu_0\left(2H^+_r\sigma\gamma_i+H^+_i(\sigma^2-\gamma_i^2)\right)-\Delta H^+_r\sigma+\Delta H^+_i\gamma_i$ and $P$ and $Q$ are given in \eref{eq:defP} and \eref{eq:defQ}, respectively.
Applying $P[H^+]\cdot\na$ and $Q[H^+]\cdot\na$ on equation \eref{eq:realGovern-3}, we obtain
\begin{equation}
  P[H^+]\cdot\na(P[H^+]\cdot\na\sigma) + P[H^+]\cdot\na(Q[H^+]\cdot\na\gamma_i) + P[H^+]\cdot\na\kappa[\gamma,H^+] ~=~ 0
  \label{eq:+PP+PQ}
\end{equation}
and
\begin{equation}
  Q[H^+]\cdot\na(P[H^+]\cdot\na\sigma) + Q[H^+]\cdot\na(Q[H^+]\cdot\na\gamma_i) + Q[H^+]\cdot\na\kappa[\gamma,H^+] ~=~ 0.
  \label{eq:+QP+QQ}
\end{equation}

Similarly, the imaginary part \eref{eq:imagGovern-2} can be written as
\begin{equation}
    -Q[H^+]\cdot\na\sigma + P[H^+]\cdot\na\gamma_i + \tau[\gamma,H^+] = 0 \quad\mbox{in}~\Omega
    \label{eq:imagGovern-3}
\end{equation}
where $\tau:=\om\mu_0\left(H^+_r(\sigma^2-\gamma_i^2)-2H^+_i\sigma\gamma_i\right)-\Delta H^+_r\gamma_i-\Delta H^+_i\sigma$ and $P$ and $Q$ are given in \eref{eq:defP} and \eref{eq:defQ}, respectively. Applying $P[H^+]\cdot\na$ and $Q[H^+]\cdot\na$ on equation \eref{eq:imagGovern-3}, we obtain
\begin{equation}
  -Q[H^+]\cdot\na(Q[H^+]\cdot\na\sigma) + Q[H^+]\cdot\na(P[H^+]\cdot\na\gamma_i) + Q[H^+]\cdot\na\tau[\gamma,H^+] ~=~ 0
  \label{eq:-QQ+QP}
\end{equation}
and
\begin{equation}
  -P[H^+]\cdot\na(Q[H^+]\cdot\na\sigma) + P[H^+]\cdot\na(P[H^+]\cdot\na\gamma_i) + P[H^+]\cdot\na\tau[\gamma,H^+] ~=~ 0.
  \label{eq:-PQ+PP}
\end{equation}

Subtracting \eref{eq:+PP+PQ} from \eref{eq:-QQ+QP} yields
\begin{equation}
  P[H^+]\cdot\na(P[H^+]\cdot\na\sigma)+Q[H^+]\cdot\na(Q[H^+]\cdot\na\sigma)
  =
  F_1[\sigma,\ep,H^+],
  \label{eq:PPQQ1}
\end{equation}
where $F_1$ is given in \eref{eq:defF1}.
Similarly, subtracting \eref{eq:+QP+QQ} from \eref{eq:-PQ+PP} yields
\begin{equation}
  P[H^+]\cdot\na(P[H^+]\cdot\na\gamma_i)+Q[H^+]\cdot\na(Q[H^+]\cdot\na\gamma_i)
  =
  F_2[\sigma,\ep,H^+]
  \label{eq:PPQQ2}
\end{equation}
with $F_2$ being given by \eref{eq:defF2}.
A direct computation shows that equation \eref{eq:PPQQ1} can be expressed as
\begin{eqnarray}
  \na\cdot\left(\left[\begin{array}{ccc}
    P_x^2+Q_x^2 & P_xP_y+Q_xQ_y & P_xP_z+Q_xQ_z \\
    P_xP_y+Q_xQ_y & P_y^2+Q_y^2 & P_xP_z+Q_xQ_z \\
    P_xP_z+Q_xQ_z & P_xP_z+Q_xQ_z & P_z^2+Q_z^2
  \end{array}\right]\na\sigma\right)
  \nonumber
  \\
  +F_0[H^+]\cdot\na\sigma
  =F_1[\sigma,\epsilon,H^+].
  \label{eq:elliptic1}
\end{eqnarray}
Since $P_x=-Q_y$ and $P_y=Q_x$, $P_xP_y+Q_xQ_y=0$ and therefore, equation \eref{eq:elliptic1} can be rewritten as
\begin{equation}
  \na\cdot\left(A[H^+]\na\sigma\right)
  +
  F_0[H^+]\cdot\na\sigma
  =
  F_1[\sigma,\ep,H^+]
  \q\mbox{in }~\Om.
\end{equation}
Similarly, \eref{eq:PPQQ2} gives
\begin{equation}
  \na\cdot\left(A[H^+]\na\epsilon\right)
  +
  F_0[H^+]\cdot\na\epsilon
  =
  F_2[\sigma,\ep,H^+]
  \q\mbox{in }~\Om.
\end{equation}

Now, it remains to prove that the matrix $A$ is positive semi-definite matrix. A direct computation shows
\begin{equation}\label{eigenvalue_A}
  \det(A-\lambda I) = -\lambda~(\lambda-(P_x^2+P_y^2))~(\lambda-(P_x^2+P_y^2+P_z^2+Q_z^2)),
\end{equation}
where $\det$ denotes the determinant.
Hence, all the eigenvalues of the matrix $A$ are non-negative.
\end{proof}


\subsection{Approximate solution}\label{subsec:Approximate solution}
Using the elliptic partial differential equation \eref{eq:poisson-sigma-epsilon} in Theorem \ref{thm:ellipticPDE}, we can compute a fairly good approximation of the true admittivity. Since the matrix $A$ in \eref{eq:poisson-sigma-epsilon} is degenerate, we need a regularization strategy.
By adding a regularization term $\rho\e_3^T\e_3$ to the matrix $A$, we can compute viscosity solution $U^\rho=(\sigma^\rho,\omega\ep^\rho)^T$ of the elliptic partial differential equation \eref{eq:poisson-sigma-epsilon}:
\begin{equation}\label{eq:poisson-U}
  \na\cdot\left((A[H^+]+\rho\e_3^T\e_3)\na U^\rho\right)
  + F_0[H^+]\cdot\na U^\rho
  =
  \left(\begin{array}{c}
    F_1[  U^\rho,H^+] \\
    F_2[  U^\rho,H^+]
  \end{array}\right)
\end{equation}
with the Dirichlet boundary condition $U_0=(\sigma_0,\omega\ep_0)^T$ on $\partial \Omega$, where $\e_3=(0,0,1)$, superposed $T$ denotes the transpose, and $\rho$ is a small positive constant.
%
Note that the matrix $A+\rho\e_3^T\e_3$ is positive definite, since the eigenvalues of the matrix $A+\rho\e_3^T\e_3$ are
\begin{eqnarray*}
  \lambda_1 &=& P_x^2+P_y^2 \\
  \lambda_2,~\lambda_3 &=& \frac{(P_x^2+P_y^2+P_z^2+Q_z^2+\rho)\pm\sqrt{(P_x^2+P_y^2+P_z^2+Q_z^2+\rho)^2-4(P_x^2+P_y^2)\rho}}{2}.
\end{eqnarray*}
By solving equation \eref{eq:poisson-U}, we can get the blurred admittivity image of the true distribution.

\section{Adjoint-based optimization method}\label{method2}
This section presents adjoint-based optimization method for finding admittivity distribution. A Newton iteration is used to find optimal solution, hence, a fairly good initial guess is required. The approximated solution in section \ref{subsec:Approximate solution} is used as an initial guess of the Newton iteration.
Let $H^+_m \in H^1(\Omega)$ be the measured data corresponding to the true admittivity $\gamma^\ast \in \mathcal{A}$; hence $H^+_m$ satisfies
$$
  \mathcal L H^+_m
  \cdot
  \f{\na \gamma^\ast}{\gamma^\ast} - i\omega\mu_0 \gamma^\ast ~H^+_m +\Delta H^+_m
  =
  0 \quad \mbox{in } \Omega.
$$
For $\gamma \in \mathcal{A}$, let $H^+[\gamma]$ be a solution of the Dirichlet problem:
\begin{equation}\label{EPTmap}
  \left\{\begin{array}{rcll}
    \mathcal L H^+[\gamma]
    \cdot
    \f{\na \gamma}{\gamma} - i\omega\mu_0\gamma ~H^+[\gamma]+\Delta H^+[\gamma]
    &=&
    0 & \mbox{in } \Omega,
    \\
    H^+ &=& H^+_m & \mbox{on } \partial\Omega.
  \end{array}\right.
\end{equation}
The equation \eref{EPTmap} has a unique solution for properly chosen $c_1$ in the definition of $\mathcal{A}$ in \eref{defA}. From now on, we assume that $c_1$ is chosen so that \eref{EPTmap} has a unique solution. Then, the map
\begin{equation}
  \gamma\in\mathcal{A} ~\mapsto~ H[\gamma]
\end{equation}
is well-defined.

We define the misfit function $J[\gamma]$ of the variable $\gamma=\sigma+i\om\ep$ by the $L^2$-norm of the difference between $H^+[\gamma]$ in \eref{EPTmap} and the measured data $H^+_m$:
\begin{equation}\label{minproblem}
  J[\gamma]=\f{1}{2}\int_{\Omega}|H^+[\gamma]- H^+_m|^2d\r.
\end{equation}
Since $J[\gamma] = \f{1}{2}\left\|H^+[\gamma]- H^+_m\right\|_{L^2(\Om)}^2$, $J[\gamma] \ge 0$ and $J[\gamma]$ has minimum 0 at $H^+[\gamma]=H^+_m$. In this minimization problem, we need to determine the Fr\'echet derivative of the misfit function $J$ with respect to the control variable $\gamma$.
Let $\widetilde{\mathcal A}$ be defined by
$$
  \widetilde{\mathcal A} = \left\{\delta\gamma\in H^2(\Omega)\cap L^\infty_{\underline{\lambda}, \overline{\lambda}}(\Om) ~|~  ~ \left. \delta\gamma\right|_{\Omega_d}= 0\right\}.
$$
The following theorem proves the Fr\'echet differentiability of $H^+[\gamma]$ under the assumption that $c_1 <1$.
\begin{theorem}\label{thm:H-diff}
Let $c_1 < 1$. For $\gamma \in \mathcal A$, the map $\gamma \mapsto H^+$ is Fr\'echet differentiable. Let $\delta \in \widetilde{\mathcal A}$ be such that $\gamma + \delta \in \mathcal A$. The Fr\'echet derivative $DH^+[\gamma](\delta)$ at $\delta$ is given by the solution $u$ of the following equation
\begin{equation}\label{eq:H-diff}
  \left\{\begin{array}{rcll}
    \mathcal Lu\cdot\f{\na \gamma}{\gamma} - i\om\mu_0\gamma u + \Delta u &=& -\left(\mathcal LH^+[\gamma]\cdot\na\left(\f{\delta}{\gamma}\right) -i\om\mu_0\delta H^+[\gamma]\right) & \mbox{in}~\Om ,\\
    u &=& 0 & \mbox{on}~\p\Om .
  \end{array}\right.
\end{equation}
\end{theorem}
\begin{proof}
First, remember that for $w \in H^1(\Omega)$,
\begin{equation} \label{sobolev}
|| w||_{L^4(\Omega)} \leq 2 |\Omega|^{1/12} ||w||_{H^1(\Omega)}.
\end{equation}
Then, defining
\begin{equation*}
  w_\delta := H^+[\gamma+\delta] - H^+[\gamma] \in H^1(\Om),
\end{equation*}
it follows from \eref{EPTmap} that
\begin{equation}\label{eq:H-gamma+delta}
\left\{\begin{array}{l}
  \mathcal L w_\delta\cdot\f{\na(\gamma+\delta)}{\gamma+\delta}-i\om\mu_0(\gamma+\delta)w_\delta+\Delta w_\delta = \\
  \qquad\qquad
  -\bigg(\mathcal LH^+[\gamma]\cdot\na\left(\f{\delta}{\gamma}\right) \frac{\gamma^2}{\gamma (\gamma + \delta)}-i\om\mu_0\delta H^+[\gamma]\bigg) \quad\mbox{in}~\Omega ,\\
  \left.w_\delta\right|_{\partial\Omega} =0.
\end{array}\right.
\end{equation}
Therefore, we have
\begin{equation}\label{ineq:w}
  \begin{array}{rl}
    \left\|w_\delta\right\|_{H^2(\Om)}
    \leq &
    \left\|\mathcal L w_\delta\cdot\f{\na(\gamma+\delta)}{\gamma+\delta}
    -i\om\mu_0(\gamma+\delta)w_\delta\right\|_{L^2(\Omega)}
    \\
    &~
    +\left\|\mathcal LH^+[\gamma]\cdot\na\left(\f{\delta}{\gamma}\right)\frac{\gamma^2}{\gamma(\gamma+\delta)} - i\om\mu_0\delta H^+[\gamma]\right\|_{L^2(\Om)} .
  \end{array}
\end{equation}
By H\"older's inequality and Sobolev embedding theorem (see (\ref{sobolev})), the first term of the right-hand side of \eref{ineq:w} can be estimated by
\begin{equation}\label{ineq:w-r1}
  \left\|\mathcal L w_\delta\cdot\f{\na(\gamma+\delta)}{\gamma+\delta}
  -i\om\mu_0(\gamma+\delta)w_\delta\right\|_{L^2(\Omega)}
  \leq
  c_1 \left\|w_\delta\right\|_{H^2(\Omega)}.
\end{equation}
Combining \eref{ineq:w} and \eref{ineq:w-r1}, we have
\begin{equation}\label{ineq:w-2}
    \left(1- c_1 \right) \left\|w_\delta\right\|_{H^2(\Om)}
    \leq
    \left\|\mathcal LH^+[\gamma]\cdot\na\left(\f{\delta}{\gamma}\right)\frac{\gamma^2}{\gamma(\gamma+\delta)} - i\om\mu_0\delta H^+[\gamma]\right\|_{L^2(\Om)}.
\end{equation}
By H\"older's inequality and Sobolev embedding theorem, \eref{ineq:w-2} can be estimated by
\begin{equation}\label{ineq:w-3}
  \left\|w_\delta\right\|_{H^2(\Om)}
  \leq
  C'\left\|\delta\right\|_{H^2(\Om)}\left\|H^+[\gamma]\right\|_{H^2(\Om)}
\end{equation}
if $c_1<1$.

Since the data difference $w_\delta$ satisfies \eref{eq:H-gamma+delta} and $u$ is the solution of equation \eref{eq:H-diff}, the difference $w_\delta-u \in H^1(\Om)$ satisfies
\begin{equation}\label{eq:w-u}
\begin{array}{r}
  \mathcal L (w_\delta-u)\cdot\f{\na\gamma}{\gamma}-i\om\mu_0\gamma(w_\delta-u)+\Delta(w_\delta-u)
  ~=~
  -\bigg(\mathcal Lw_\delta\cdot\na\left(\f{\delta}{\gamma}\right) \frac{\gamma^2}{\gamma (\gamma + \delta)}
  \\
  -i\om\mu_0\delta w_\delta + \mathcal L H^+[\gamma] \cdot\na\left(\f{\delta}{\gamma}\right)  \frac{\gamma}{\gamma + \delta} \bigg).
  \end{array}
\end{equation}
From the standard estimation of the Poisson equation, we have
\begin{equation}\label{ineq:w-u}
  \begin{array}{l}
    \|w_\delta-u\|_{H^2(\Om)}
    \leq
    \left\|
      \mathcal L (w_\delta-u)\cdot\f{\na\gamma}{\gamma}-i\om\mu_0\gamma(w_\delta-u)
    \right\|_{L^2(\Omega)}
    \\ \qquad\qquad
    +\left\|
      \mathcal Lw_\delta\cdot\na\left(\f{\delta}{\gamma}\right)\frac{\gamma^2}{\gamma(\gamma+\delta)}
    -i\om\mu_0\delta w_\delta + \mathcal L H^+[\gamma]\cdot\na\left(\f{\delta}{\gamma}\right)\frac{\gamma}{\gamma+\delta}
    \right\|_{L^2(\Omega)}.
  \end{array}
\end{equation}
Again, by  H\"older's inequality and Sobolev embedding theorem, the first term of the right-hand side of \eref{ineq:w-u} can be estimated by
\begin{equation}\label{ineq:w-u-r1}
    \left\|
      \mathcal L (w_\delta-u)\cdot\f{\na\gamma}{\gamma}-i\om\mu_0\gamma(w_\delta-u)
    \right\|_{L^2(\Omega)}
    \leq
    c_1 \left\|w_\delta-u\right\|_{H^2(\Omega)}.
\end{equation}
Combining \eref{ineq:w-u} and \eref{ineq:w-u-r1}, we have
\begin{equation}\label{ineq:w-u-2}
  \begin{array}{l}
    \left(1- c_1 \right)\left\|w_\delta-u\right\|_{H^2(\Omega)}
    \leq
    \\
    \qquad
    \left\|
      \mathcal Lw_\delta\cdot\na\left(\f{\delta}{\gamma}\right)\frac{\gamma^2}{\gamma(\gamma+\delta)} - i\om\mu_0\delta w_\delta + \mathcal L H^+[\gamma]\cdot\na\left(\f{\delta}{\gamma}\right)\frac{\gamma}{\gamma+\delta}
    \right\|_{L^2(\Omega)}.
  \end{array}
\end{equation}
%
%
By  H\"older's inequality and Sobolev embedding theorem, \eref{ineq:w-u-2} can be estimated by
\begin{equation}\label{ineq:w-u-3}
  \left\|w_\delta-u\right\|_{H^2(\Om)}
  \leq
  C_1\left\|\delta\right\|_{H^2(\Omega)} \left\|w_\delta\right\|_{H^2(\Omega)}
  +
  C_2\left\|\delta\right\|_{H^2(\Omega)} \left\|H^+[\gamma]\right\|_{H^2(\Omega)}
\end{equation}
if $c_1<1$.

By inequalities \eref{ineq:w-3} and \eref{ineq:w-u-3}, it follows that
\begin{equation}
  \left\|w_\delta-u\right\|_{H^2(\Om)}
  \leq
  C_1'\left\|\delta\right\|_{H^2(\Omega)}^2 \left\|H^+[\gamma]\right\|_{H^2(\Omega)}
  +
  C_2\left\|\delta\right\|_{H^2(\Omega)} \left\|H^+[\gamma]\right\|_{H^2(\Omega)}.
\end{equation}
%
%
Thus,
\begin{equation*}
  \f{\|H^+[\gamma+\delta]-H^+[\gamma]-u\|_{H^2(\Om)}}{\|\delta\|_{H^2(\Om)}} \rightarrow 0 \q \mbox{as}\q \|\delta\|_{H^2(\Om)}\rightarrow 0.
\end{equation*}
Hence, $u$ is the Fr\'echet derivative of $H^+[\gamma]$ at $\delta$, that is,  $DH^+[\gamma](\delta) =u$.
\end{proof}

The following theorem expresses the Fr\'echet derivative of $J[\gamma]$.
\begin{theorem}\label{thm:J-diff}
For $\gamma=\sigma+i\om\epsilon \in \mathcal{A}$, the Fr\'echet derivative of $J[\gamma]$ at $\delta\in\widetilde{\mathcal A}$ being such that $\gamma + \delta \in \mathcal A$ is given by
\begin{eqnarray}\label{eq:DJgammadelta}
  DJ[\gamma](\delta) &=& \Re \int_\Om\delta \left(\f{1}{\gamma}\na\cdot\left(p\mathcal{L}H^+[\gamma]\right)
  +i\om\mu_0H^+[\gamma]p\right)d\r,
\end{eqnarray}
where $p$ is the solution of the adjoint problem:
\begin{equation}\label{eq:adjoint-problem}
  \left\{\begin{array}{rcll}
  \Delta p+ \mathcal{L}\cdot\left(p\f{\na\gamma}{\gamma}\right) - i\omega\mu_0\gamma p
  &=& \overline{H^+[\gamma]-H^+_m} & \mbox{in}~ \Omega, \\
  p &=& 0 & \mbox{on}~ \partial\Omega.
  \end{array}\right.
\end{equation}
\end{theorem}
\begin{proof}
To compute the Fr\'echet derivative of $J[\gamma]$, we consider the perturbation $J[\gamma+\delta]-J[\gamma]$:
\begin{eqnarray}\label{eq_inverse3D-2}
  J[\gamma+\delta]-J[\gamma]
  &=& \frac{1}{2}\int_\Omega\left|  H^+[\gamma+\delta]
  - H^+_m \right|^2 d\r - \frac{1}{2}\int_\Omega\left|  H^+[\gamma]
  - H^+_m \right|^2 d\r \nonumber \\
  &=& \Re\int_\Om w_\delta \overline{(H^+[\gamma]-H^+_m)}~d\r + \f{1}{2}\int_\Omega w_\delta^2 ~d\r,
\end{eqnarray}
where $w_\delta= H^+[\gamma+\delta]- H^+[\gamma]$. So,
\begin{equation}\label{eq:w-1}
  \left|J[\gamma+\delta]-J[\gamma]-\Re\int_\Om w_\delta \overline{(H^+[\gamma]-H^+_m)}~
  d\r\right|
  =
  \left|\f{1}{2}\int_{\Om}w_\delta^2~d\r\right|.
\end{equation}
By \eref{ineq:w-3},
\begin{equation*}
  \left|\f{1}{2}\int_{\Om}w_\delta^2~d\r\right|
  ~=~
  \f{1}{2}\left\|w_\delta\right\|_{L^2(\Om)}^2
  ~\leq~
  C\left\|\delta\right\|^2_{H^2(\Om)}\left\|H^+[\gamma]\right\|^2_{H^2(\Om)}.
\end{equation*}
Thus,
\begin{equation*}
  \lim_{\delta\rightarrow0}\f{\left|J[\gamma+\delta]-J[\gamma]-\Re\int_\Om w_\delta \overline{(H^+[\gamma]-H^+_m)}~d\r\right|}{\|\delta\|_{H^2(\Om)}}
  =
  0.
\end{equation*}
Therefore, the Fr\'echet derivative $DJ[\gamma](\delta)$ is $\displaystyle \Re\int_\Om w_\delta \overline{(H^+[\gamma]-H^+_m)}~d\r$.
Using the adjoint problem (\ref{eq:adjoint-problem}) with the homogeneous Dirichlet boundary condition, we get
\begin{equation*}\label{eq_inverse3D-4}
  DJ[\gamma](\delta)
  =
  \Re\int_{\Om} w_\delta\left(\mathcal{L}\cdot\left(p\f{\na\gamma}{\gamma}\right) - i\omega\mu_0\gamma p + \Delta p\right)~d\r.
\end{equation*}
On integrating by parts, it follows that
\begin{equation*}
  \int_\Om w_\delta\Delta p~d\r
  =
  \int_{\p\Om}w_\delta \f{\p p}{\p \n} ds-\int_\Om \na w_\delta\cdot\na p~d\r
  =
  -\int_{\p\Om} p \f{\p w_\delta}{\p \n} ds+\int_\Om p\Delta w_\delta~d\r.
\end{equation*}
Moreover,
\begin{equation*}
  \int_\Om w_\delta~\mathcal{L}\cdot\left(\f{\na\gamma}{\gamma}p\right)d\r
  =
  \int_{\p\Om}\mathcal{L}\cdot\left(\f{\na\gamma}{\gamma}p\right)\f{\p w_\delta}{\p \n} ds-\int_\Om \left(\f{\na\gamma}{\gamma}p\right)\cdot\mathcal{L}w_\delta~d\r .
\end{equation*}
Hence,
\begin{equation*}
  DJ[\gamma](\delta)
  =
  \Re\int_\Om p \left(\mathcal{L}w_\delta\cdot\f{\na\gamma}{\gamma}-i\omega\mu_0\gamma~w_\delta+\Delta w_\delta\right)~d\r.
\end{equation*}
Note that $w_\delta$ satisfies the following identity:
\begin{equation}\label{eq_inverse3D-3}
  \mathcal{L}w_\delta\cdot\f{\na\gamma}{\gamma}-i\omega\mu_0\gamma w_\delta + \Delta w_\delta
  =
  -\mathcal{L}H^+[\gamma+\delta]
  \cdot
  \na\left(\f{\delta}{\gamma}\right)+i\omega\mu_0\delta H^+[\gamma+\delta].
\end{equation}
So,
\begin{equation*}
  DJ[\gamma](\delta) =  \Re \int_\Om p \left(-\mathcal{L}H^+[\gamma]
  \cdot
  \na\left(\f{\delta}{\gamma}\right)
  +i\omega\mu_0\delta H^+[\gamma]\right)d\r.
\end{equation*}
Since $\mathcal{L}H^+[\gamma]\cdot\na\left(\f{\delta}{\gamma}\right)=\na\cdot\left(\f{\delta}{\gamma}\mathcal{L}H^+[\gamma]\right)-\f{\delta}{\gamma}(\na\cdot\mathcal{L}H^+[\gamma])$,
\begin{eqnarray*}
   \int_\Om p \left(-\mathcal{L}H^+[\gamma]\cdot\na\left(\f{\delta}{\gamma}\right)\right)d\r
  &=& -\int_\Om p ~\na\cdot\left(\f{\delta}{\gamma}\mathcal{L}H^+[\gamma]\right)d\r + \int_\Om p ~\f{\delta}{\gamma}(\na\cdot\mathcal{L}H^+[\gamma])d\r
  \\
  &=&
  -\int_{\p\Om} p\left(\f{\delta}{\gamma}\mathcal{L}H^+[\gamma]\right)\cdot \n ~ds+\int_\Om \na p\cdot\left(\f{\delta}{\gamma}\mathcal{L}H^+[\gamma]\right)d\r
  \\ &~&
  +\int_\Om p ~\f{\delta}{\gamma}(\na\cdot\mathcal{L}H^+[\gamma])d\r \\
  &=& \int_\Om \f{\delta}{\gamma}\left(\na p\cdot\mathcal{L}H^+[\gamma]+p\na\cdot\mathcal{L}H^+[\gamma]\right)d\r\\
  &=& \int_\Om \f{\delta}{\gamma}\na\cdot\left(p\mathcal{L}H^+[\gamma]\right)d\r.
\end{eqnarray*}
Therefore,
\begin{equation*}
  DJ[\gamma](\delta) = \Re \int_\Om\delta \left(\f{1}{\gamma}\na\cdot\left(p\mathcal{L}H^+[\gamma]\right)
  +i\om\mu_0H^+[\gamma]p\right)d\r,
\end{equation*}
which completes the proof.
\end{proof}

It is worth mentioning that the smallness assumption on the bound $c_1$ defined in (\ref{defA}) ensures the well-posedness of (\ref{eq:DJgammadelta}) with homogeneous Dirichlet boundary condition.

In the next lemma, we rewrite
the adjoint problem \eref{eq:adjoint-problem} as a second-order elliptic partial differential equation.
\begin{lemma}\label{thm:adjoint-PDE}
For $\gamma=\sigma+i\om\epsilon$, the adjoint problem \eref{eq:adjoint-problem} can be rewritten as
\begin{equation}\label{eq:adjoint-PDE}
  \Delta p + G[\gamma]\cdot\na p -(i\om\mu_0\gamma+\Delta\log\gamma) p = \overline{H^+[\gamma]-H^+_m} \quad \mbox{in } \Omega
\end{equation}
with  the Dirichlet boundary condition $p=0$ on $\partial \Omega$,
where
\begin{equation*}
  G = -\left(\f{\p \log\gamma}{\p x}+i\f{\p \log\gamma}{\p y}, -i\f{\p \log\gamma}{\p x}+\f{\p \log\gamma}{\p y}, \f{\p \log\gamma}{\p z}\right).
\end{equation*}

\begin{proof}
Denote by $v:=\log\gamma$. Since the linear operator $\mathcal L$ is given by
\begin{equation*}
  \mathcal L = \left(-\f{\p}{\p x}+i\f{\p}{\p y}, -i\f{\p}{\p x}-\f{\p}{\p y},-\f{\p}{\p z}\right) = -\na+i\left(\f{\p}{\p y},-\f{\p}{\p x},0\right),
\end{equation*}
we obtain
\begin{eqnarray}\label{eq:linear_operator1}
  \mathcal L \cdot \left(\f{\na\gamma}{\gamma}p\right) &=& \mathcal L \cdot (p\na v) = -\na\cdot(p\na v) + i\left(\f{\p}{\p y}, -\f{\p}{\p x}, 0\right)\cdot(p\na v)\nonumber\\
  &=& -(p\Delta v+\na v\cdot\na p)+i\left(\f{\p}{\p y}\left(p \f{\p v}{\p x}\right)-\f{\p}{\p x}\left(p \f{\p v}{\p y}\right)\right)\nonumber\\
  &=& -(\Delta v)p + G[\gamma]\cdot\na p.
\end{eqnarray}
Hence, if we substitute \eref{eq:linear_operator1} into the adjoint problem \eref{eq:adjoint-problem}, we get the second-order elliptic partial differential equation \eref{eq:adjoint-PDE} and the proof is complete.
\end{proof}
\end{lemma}

\section{Newton-type reconstruction algorithm}\label{algorithm}
In the section \ref{method2}, the Fr\'echet differentiability of the discrepancy functional $J$ is proven. To find $\gamma\in\mathcal A$ such that $J[\gamma] = 0$, we apply the Newton method. The Newton method starts from the linearization of the functional $J$:
\begin{equation}\label{eq:newton_linearization}
  J[\gamma+h] \approx J[\gamma] + DJ[\gamma](h),
\end{equation}
where $h\in\widetilde{\mathcal A}$ and $\gamma, \gamma+h \in \mathcal A$. Let $\gamma_n$ be the $n$-th iteration. Given $\gamma_n$, Newton's method seeks to find $h_n$ such that
\begin{equation}\label{eq:newton_zero}
  J[\gamma_n] + DJ[\gamma_n](h_n)=0.
\end{equation}
If we update the iteration as $\gamma_{n+1}=\gamma_n+h_n$, $J[\gamma_{n+1}] \approx 0$ by \eref{eq:newton_linearization} and \eref{eq:newton_zero}.  Note that $DJ[\gamma](kh)=kDJ[\gamma](h)$ for any real number $k$ by \eref{eq:DJgammadelta}. The following lemma shows how to find $h_n$.
\begin{lemma}\label{thm:Newton_step}
For given $\gamma_n$, $h_n = -\f{J[\gamma_n]}{DJ[\gamma_n](f_n)}f_n$ satisfies  \eref{eq:newton_zero} for any $f_n\in\widetilde{\mathcal A}$ such that $DJ[\gamma_n](f_n)\ne0$.
\begin{proof}
Note that $J[\gamma_n]$ and $DJ[\gamma_n](f_n)$ are real numbers. So, if we substitute $h_n = -\f{J[\gamma_n]}{DJ[\gamma_n](f_n)}f_n$ into \eref{eq:newton_zero}, then we obtain
\begin{equation*}
  J[\gamma_n]+DJ[\gamma_n]\left(-\f{J[\gamma_n]}{DJ[\gamma_n](f_n)}f_n\right)
  =
  J[\gamma_n]-\f{J[\gamma_n]}{DJ[\gamma_n](f_n)}DJ[\gamma_n](f_n)=0.
\end{equation*}
Hence, for any $f_n\in\widetilde{\mathcal A}$ such that $DJ[\gamma_n](f_n)\ne0$, $h_n = -\f{J[\gamma_n]}{DJ[\gamma_n](f_n)}f_n$ is the solution of  \eref{eq:newton_zero}.
\end{proof}
\end{lemma}

Lemma \ref{thm:Newton_step} proves that we can make the step size $h_n$ of \eref{eq:newton_zero} if we choose the function $f_n\in\widetilde{\mathcal A}$ such that $DJ[\gamma_n](f_n)\ne0$. Equation \eref{eq:DJgammadelta} shows that $DJ[\gamma_n](f_n)$ can be represented by $L_2$ inner product, $DJ[\gamma_n](f_n)=\Re<f_n,\overline{g_n}>$, where $g_n:=\f{1}{\gamma_n}\na\cdot\left(p_n\mathcal{L}H^+[\gamma_n]\right)+i\om\mu_0H^+[\gamma_n]p_n$. Note that $g_n$ is computed from given $\gamma_n$ and the solution of the adjoint problem \eref{eq:adjoint-problem} $p_n$. If we choose $f_n = \overline{g_n}$, then $DJ[\gamma_n](f_n) = ||\overline{g_n}||_2^2=||g_n||_2^2$. In that case, $DJ[\gamma_n](f_n) \ne0$ unless $g_n=0$ and the step size $h_n$ becomes
$$
  h_n =-\f{J[\gamma_n]}{DJ[\gamma_n](f_n)}f_n=-\f{J[\gamma_n]}{||g_n||_2^2}\overline{g_n}.
$$
So, the Newton iteration algorithm is given by
\begin{equation}
  \gamma_{n+1} = \gamma_n-\f{J[\gamma_n]}
  {||g_n||_2^2}\overline{g_n},
\end{equation}
where $g_n=\f{1}{\gamma_n}\na\cdot\left(p_n\mathcal{L}H^+[\gamma_n]\right)+i\om\mu_0H^+[\gamma_n]p_n$.
It is worth emphasizing that the Newton method guarantees the convergence only when the initial guess $\gamma_0$ is close enough to the true solution.


To find a good initial guess for $\gamma$, we use an iteration scheme to solve \eref{eq:poisson-U} with small regularization parameter $\rho$:
\begin{equation}\label{eq:newton-poisson-U}
  \na\cdot\left((A[H^+]+\rho\e_3^T\e_3)\na U_k^\rho\right)
  + F_0[H^+]\cdot\na U_k^\rho
  =
  \left(\begin{array}{c}
    F_1[  U_{k-1}^\rho,H^+] \\
    F_2[  U_{k-1}^\rho,H^+]
  \end{array}\right).
\end{equation}
Based on the above iteration scheme, we develop the following reconstruction algorithm.

\begin{enumerate}
  \item[\it Step 1.] For Given data $H^+$, compute the matrix $A[H^+]$ in \eref{eq:poisson-U}.
  \item[\it Step 2.] From the initial guess $U_0^\rho=(\sigma_0,\omega\ep_0)^T$, update the vector $U_k^\rho=(\sigma_k,\omega\ep_k)^T$ by solving the semi-elliptic PDE \eref{eq:newton-poisson-U} with the Dirichlet boundary condition $(\sigma_k,\omega\ep_k)=(\sigma^\ast,\omega\ep^\ast)$ on $\p\Om$, where $\sigma^\ast$ and $\ep^\ast$ are the true values.
  \item[\it Step 3.] For a given tolerance $\varepsilon_1$, iterate {\it Step 2} until $||\gamma^k-\gamma^\ast||\leq\varepsilon_1$, where $\gamma^k=\sigma_k+i\om\ep_k$ from $U_k^\rho=(\sigma_k,\omega\ep_k)^T$.  
      Result of the iteration $U_k^\rho=(\sigma_k,\omega\ep_k)^T$ defines the initial guess $\gamma_0=\sigma_k+i\omega\ep_k$ for the next step.
  \item[\it Step 4.] Compute $H^+[\gamma_n]$ from the given $\gamma_n$, for $n\geq0$. From t\eref{EPTmap}, the following equation for $H^+$ can be obtained:
      \begin{eqnarray}\label{eq:newton-EPTmap}
        \Delta H^+[\gamma_n] +G[\gamma_n]\cdot\na H^+[\gamma_n] - i\omega\mu_0\gamma_n ~H^+[\gamma_n]
        =
        0
      \end{eqnarray}
      with $H^+=H^+_m$ on $\partial \Omega$ and $G[\gamma_n]$ being the vector field in \eref{eq:adjoint-PDE}.
  \item[\it Step 5.] Compute the functional $J[\gamma_n]=\f{1}{2}\int_{\Omega}|H^+[\gamma_n]- H^+_m|^2d\r$ and the  function $g_n$ given by
      \begin{eqnarray}\label{eq:newton-adjoint}
        g_n=\f{1}{\gamma_n}\na\cdot\left(p_n\mathcal{L}H^+[\gamma_n]\right)+i\om\mu_0H^+[\gamma_n]p_n
      \end{eqnarray}
      by solving the following adjoint problem for $p_n$:
      \begin{equation*}
        \Delta p_n + G[\gamma_n]\cdot\na p_n -(i\om\mu_0\gamma_n+\Delta\log\gamma_n) p_n = \overline{H^+[\gamma_n]-H^+_m} \quad \mbox{in } \Omega
      \end{equation*}
      with $p_n=0$ on $\partial \Omega$.
  \item[\it Step 6.] Update $\gamma_n$:
      \begin{equation}\label{eq:newton}
        \gamma_{n+1} = \gamma_n-\f{J[\gamma_n]}{||g_n||_2^2}\overline{g_n}
      \end{equation}
      from $J[\gamma_n]$ and $g_n$ from  {\it Step 5}.
  \item[\it Step 7.] For a given tolerance $\varepsilon_2$, repeat from {\it Step 4} to {\it Step 6} until $||\gamma_n-\gamma^\ast||\leq\varepsilon_2$.
\end{enumerate}


\section{Numerical simulations}
In this section, we will present numerical simulation results from two models to validate the proposed algorithm. In the first model, we set the domain $\Om$ to be a cylindrical model where the admittivity distribution does not change along the $z$-direction. Figure \ref{fig:model1} shows the simulation model, the conductivity values $\sigma$, and the relative permittivity values $\epsilon/\tilde{\epsilon}$ in the domain, where $\tilde{\epsilon}=8.85\times10^{-12}[F/m]$ is the permittivity of free space.
Figure \ref{fig:data} shows the real and imaginary parts of the given data, $H^+[\gamma^\ast]$ in slice $\Om_0=\Om\cap\{z=0\}$, where $\gamma^\ast$ is the true admittivity distribution.
\begin{figure}
  \includegraphics[width=1.0\textwidth]{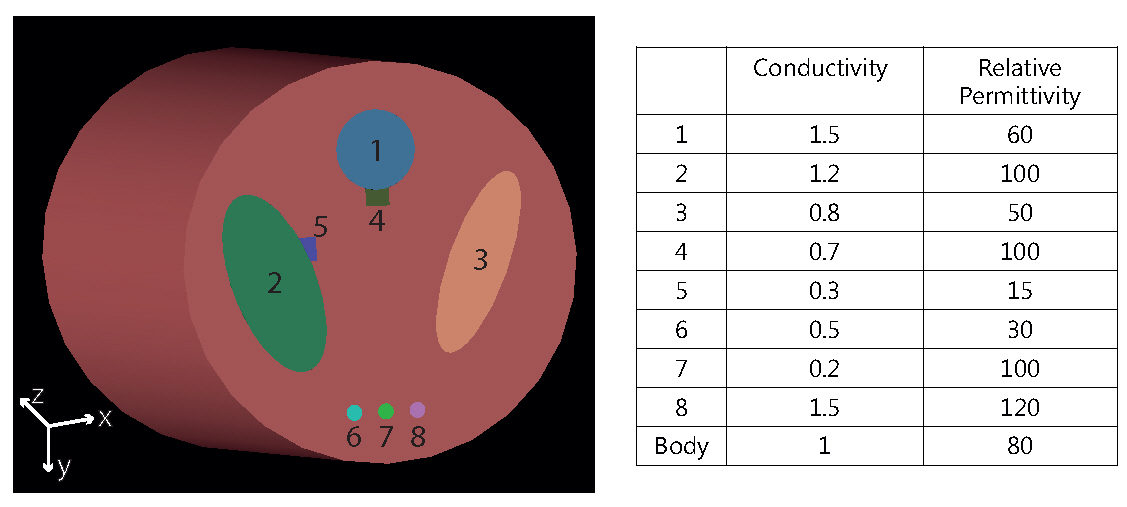}
  \caption{model configuration (left) and table of the value of electrical property (right).}
  \label{fig:model1}
\end{figure}
\begin{figure}
  \begin{tabular}{cc}
    \includegraphics[width=0.45\textwidth]{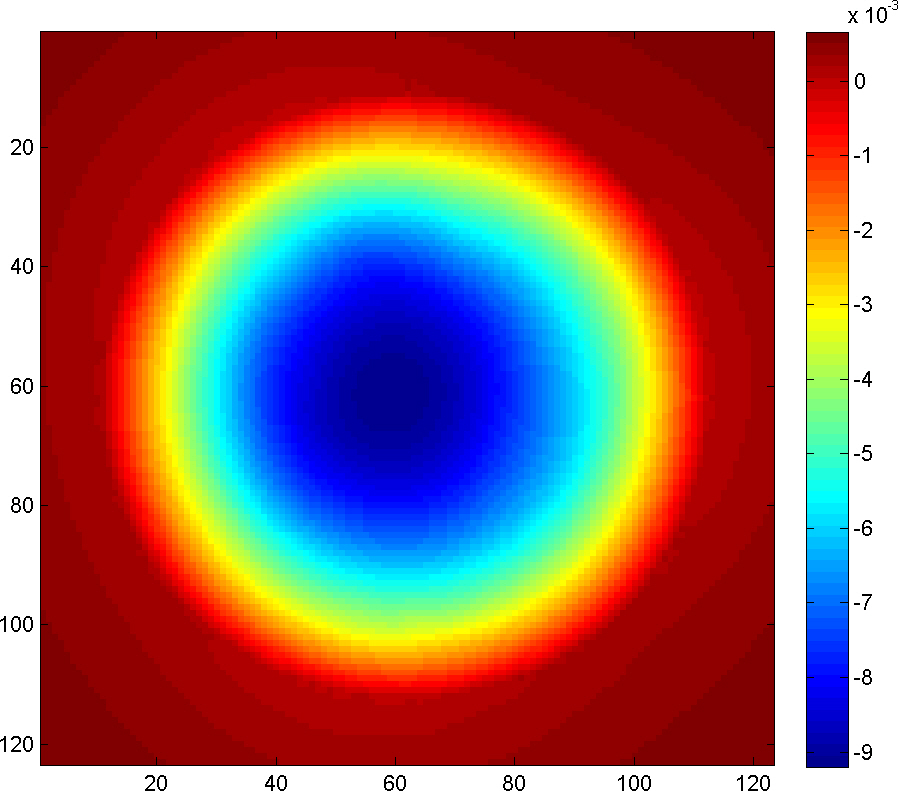}
    &
    \includegraphics[width=0.45\textwidth]{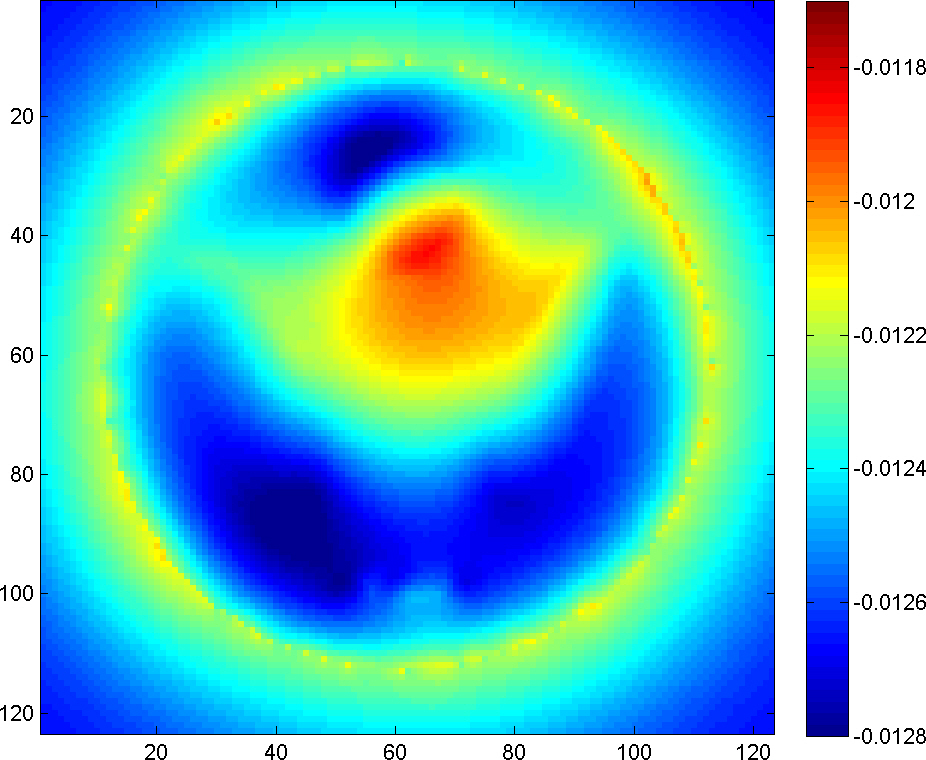}
    \\
    (a) $\Re\left(H^+[\gamma^\ast]\right)$ & (b) $\Im\left(H^+[\gamma^\ast]\right)$
  \end{tabular}
  \caption{Real part and imaginary part of the given data in the slice $\Om_0$; (a) The real part of the data $H^+[\gamma^\ast]$, (b) The imaginary part of the data $H^+[\gamma^\ast]$.}
  \label{fig:data}
\end{figure}

In subsection \ref{method1}, we proved that the solution of \eref{eq:poisson-sigma-epsilon} is the blurred approximation of true admittivity. However,  \eref{eq:poisson-sigma-epsilon} is degenerate since the diffusion matrix $A[H^+]$ is singular. So, we modified \eref{eq:poisson-sigma-epsilon} to \eref{eq:newton-poisson-U} by adding the regularization term $\rho\e_3^T\e_3$. Figure \ref{fig:degenerate} shows the determinant of $A$ and $A+\rho\e_3^T\e_3$, where $\rho$ is $5\%$ of the maximum of $A_{33}$, $P_z^2+Q_z^2$ in \eref{eq:poisson-sigma-epsilon}.
\begin{figure}
  \begin{tabular}{cc}
    \includegraphics[width=0.45\textwidth]{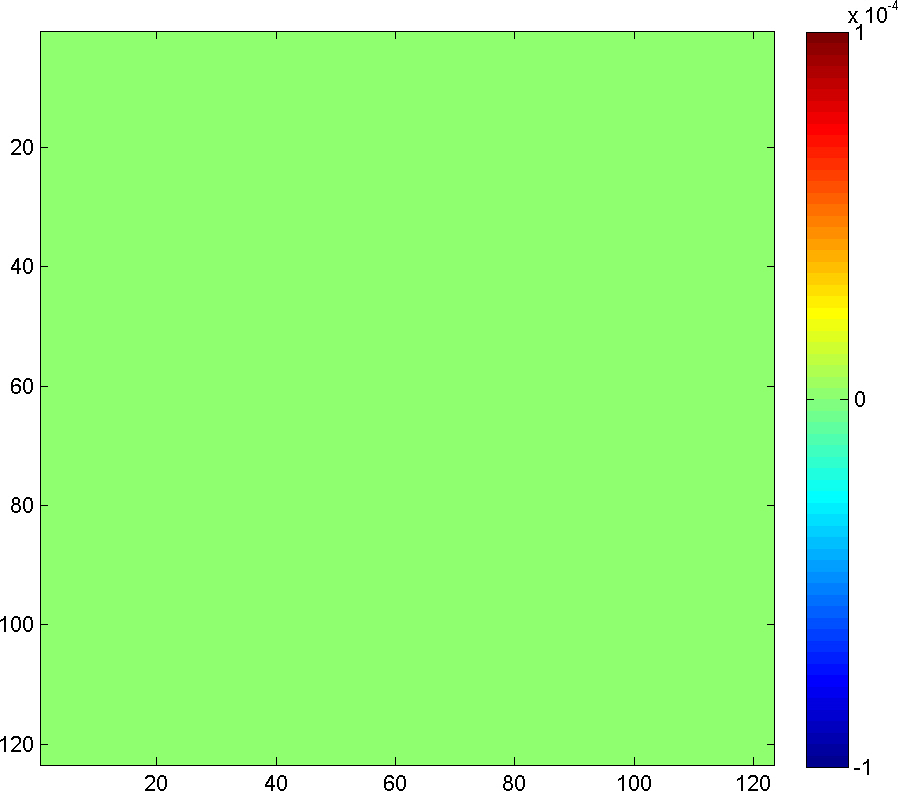}
    &
    \includegraphics[width=0.45\textwidth]{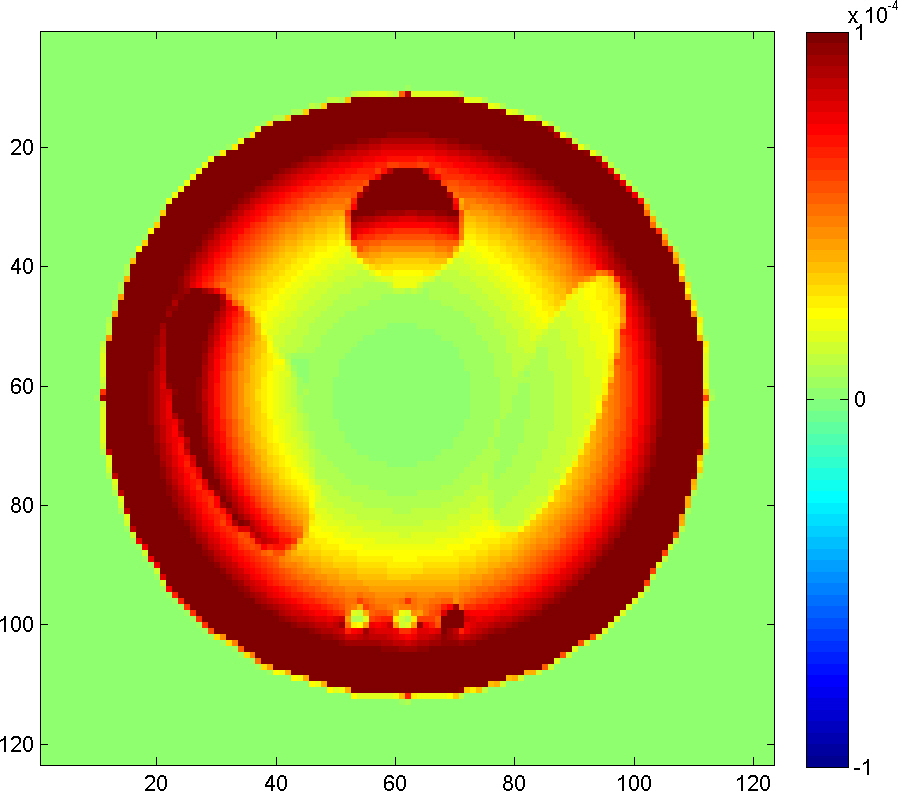}
    \\
    (a) $\mbox{det}\left(A[H^+]\right)$ & (b) $\mbox{det}\left(A[H^+]+\rho\e_3^T\e_3\right)$
  \end{tabular}
  \caption{(a) Image of the determinant of the matrix $A[H^+]$ in the slice $\Om_0$. (b) Image of the determinant of the matrix $A[H^+]+\rho\e_3^T\e_3$ in the slice $\Om_0$, where $\rho$ is $5\%$ of the maximum of $A_{33}$, $P_z^2+Q_z^2$ in \eref{eq:poisson-sigma-epsilon}.}
  \label{fig:degenerate}
\end{figure}

Figure \ref{fig:degenerate} explains that the regularized semi-elliptic PDE is also degenerate near $l=\{(0,0,z) ~|~ z\in\mathbb{R}\}$. To avoid this, we segmented subdomain $D$ near $l$, as shown in Figure \ref{fig:sElliptic}. In the subdomain $\Om\backslash D$, we applied the iteration method \eref{eq:newton-poisson-U}. Figure \ref{fig:sElliptic} illustrates solutions of  \eref{eq:newton-poisson-U}, $U_k=(\sigma_k, \om\ep_k)$, in $\Om\backslash D$ with various iteration numbers $k$. We set the initial values to be constant: $\sigma_0=1$ and $\om\ep_0=0$.
In order to check the convergence and the accuracy of the proposed algorithm \eref{eq:newton-poisson-U}, we plotted $||\gamma_k-\gamma^\ast||_2$ and $||\gamma_k-\gamma_{k-1}||_2$ with $k=1, 2, \cdots, 10$ in Figure \ref{fig:sElliptic_plot}.
\begin{figure}
  \begin{tabular}{ccccc}
    \begin{minipage}{0.8cm}$$\sigma$$\end{minipage}
    &
    \begin{minipage}{2.4cm}\includegraphics[width=2.4cm]{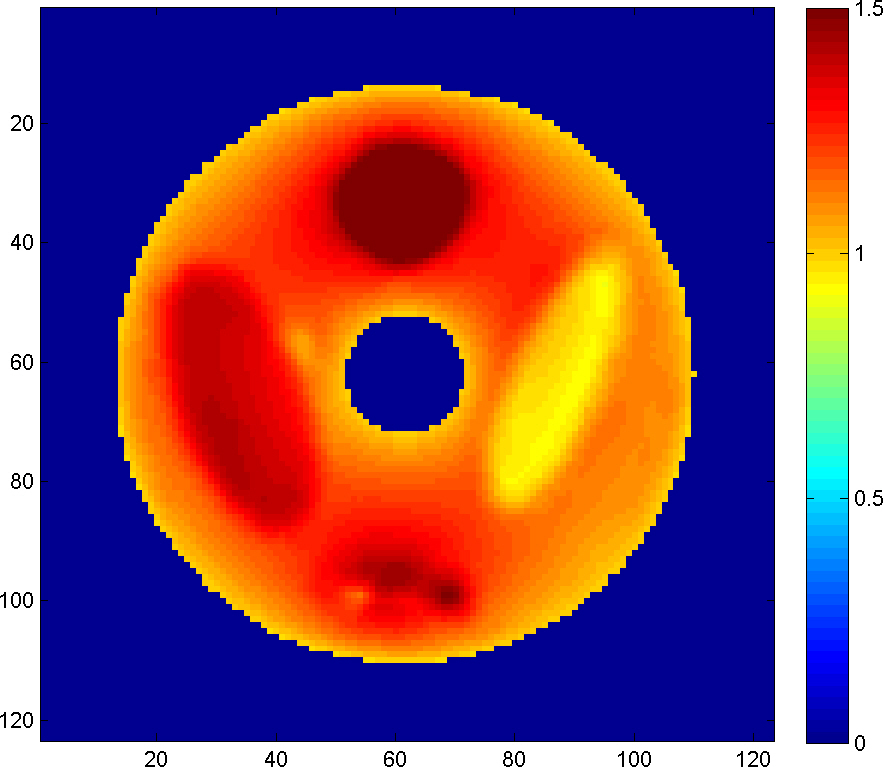}\end{minipage}
    &
    \begin{minipage}{2.4cm}\includegraphics[width=2.4cm]{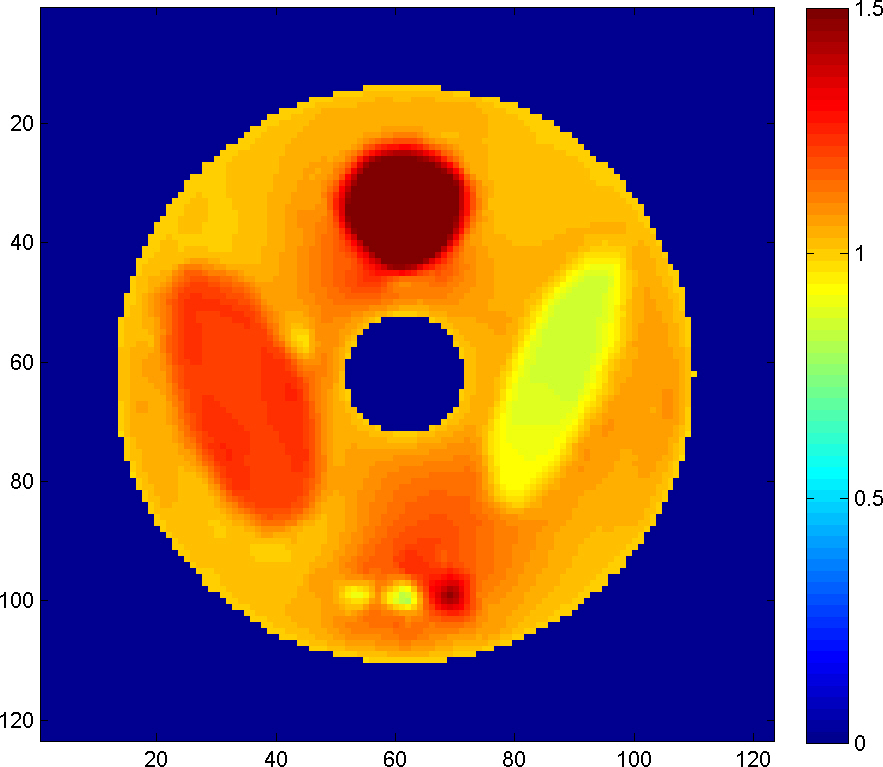}\end{minipage}
    &
    \begin{minipage}{2.4cm}\includegraphics[width=2.4cm]{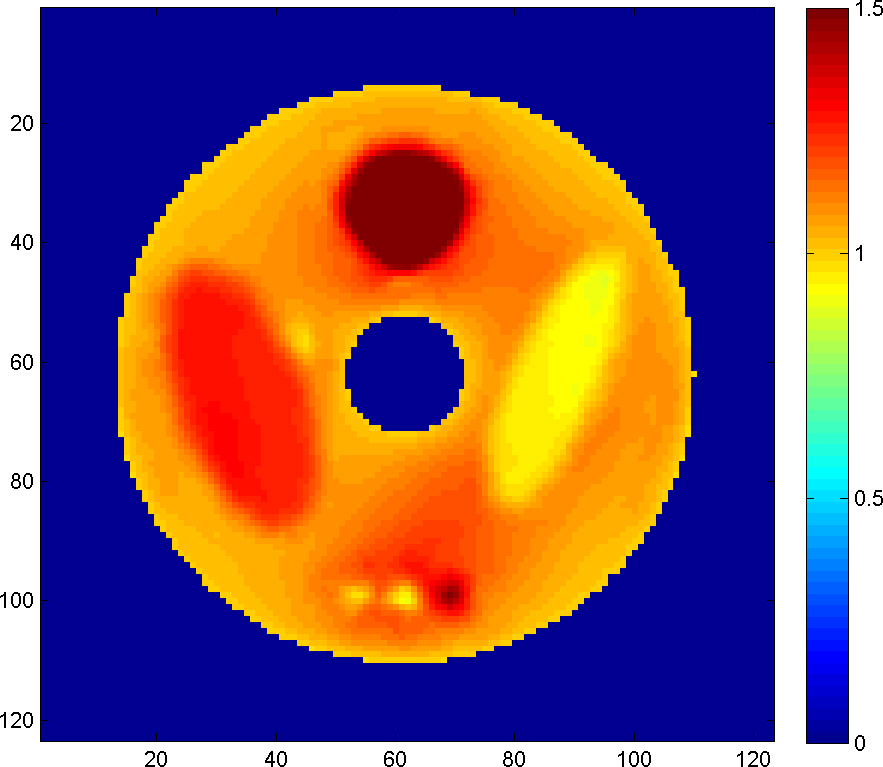}\end{minipage}
    &
    \begin{minipage}{2.4cm}\includegraphics[width=2.4cm]{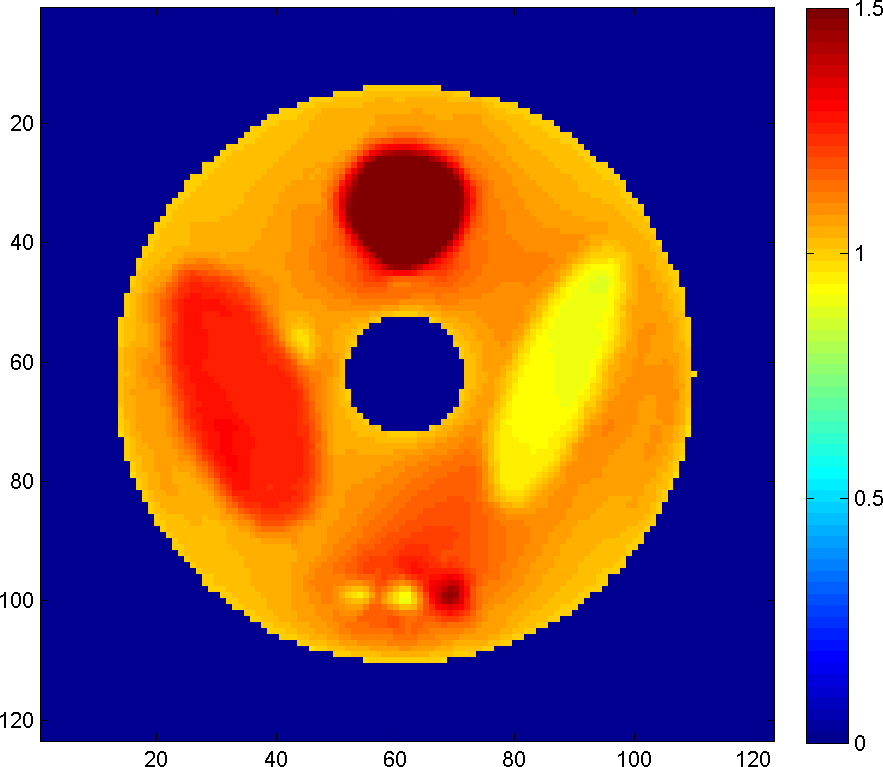}\end{minipage}
    \\
    ~ & (a) $k=1$ & (b) $k=3$ & (c) $k=5$ & (d) $k=10$
    \\
    ~
    \\
    \begin{minipage}{0.8cm}$$\epsilon/\epsilon_0$$\end{minipage}
    &
    \begin{minipage}{2.4cm}\includegraphics[width=2.4cm]{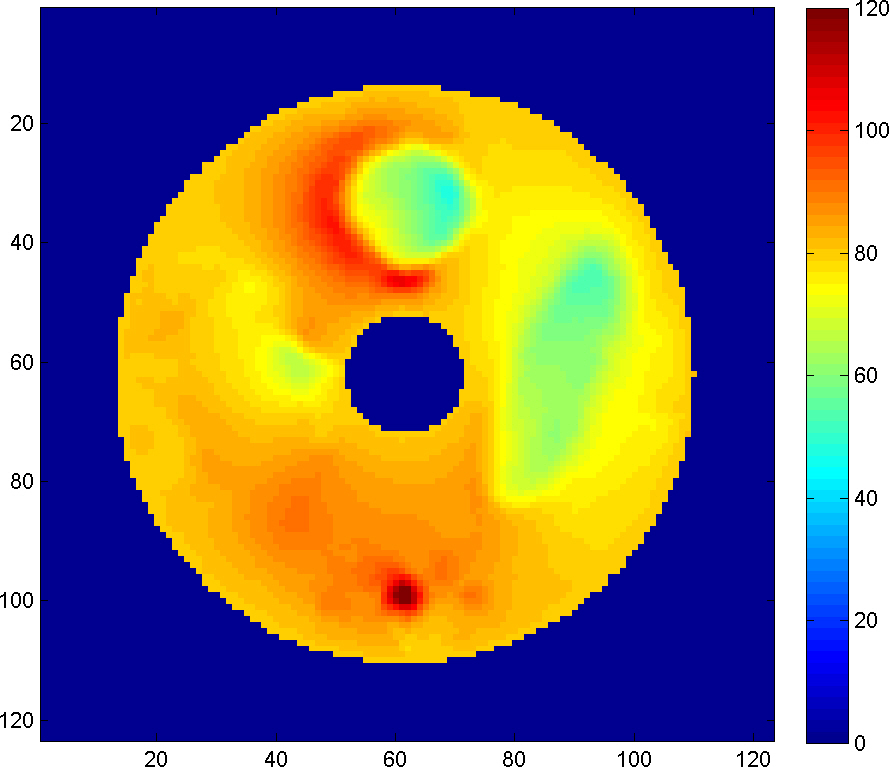}\end{minipage}
    &
    \begin{minipage}{2.4cm}\includegraphics[width=2.4cm]{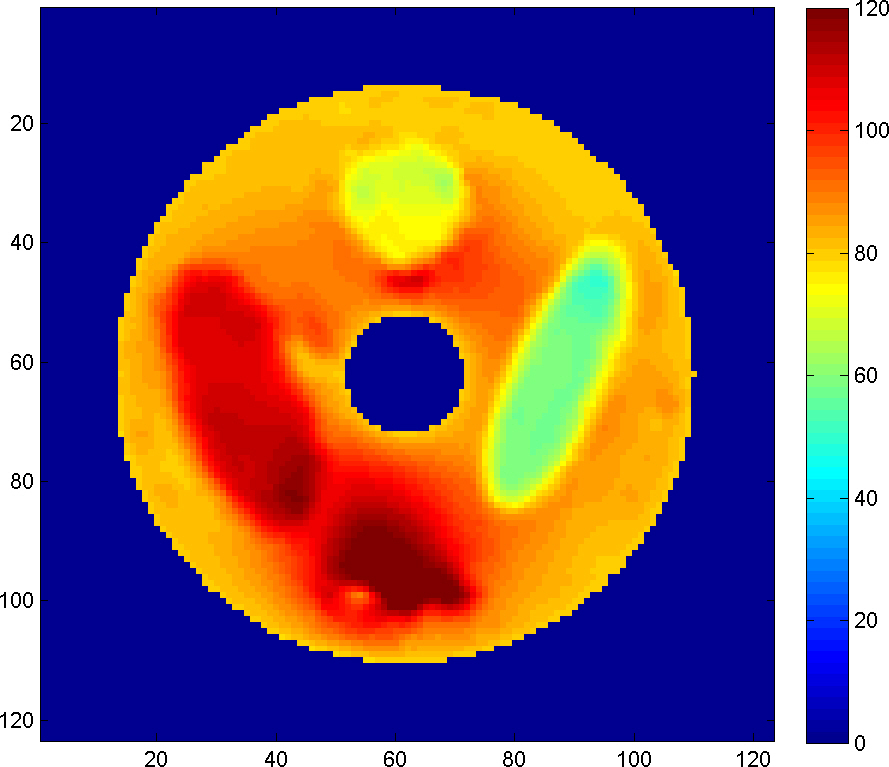}\end{minipage}
    &
    \begin{minipage}{2.4cm}\includegraphics[width=2.4cm]{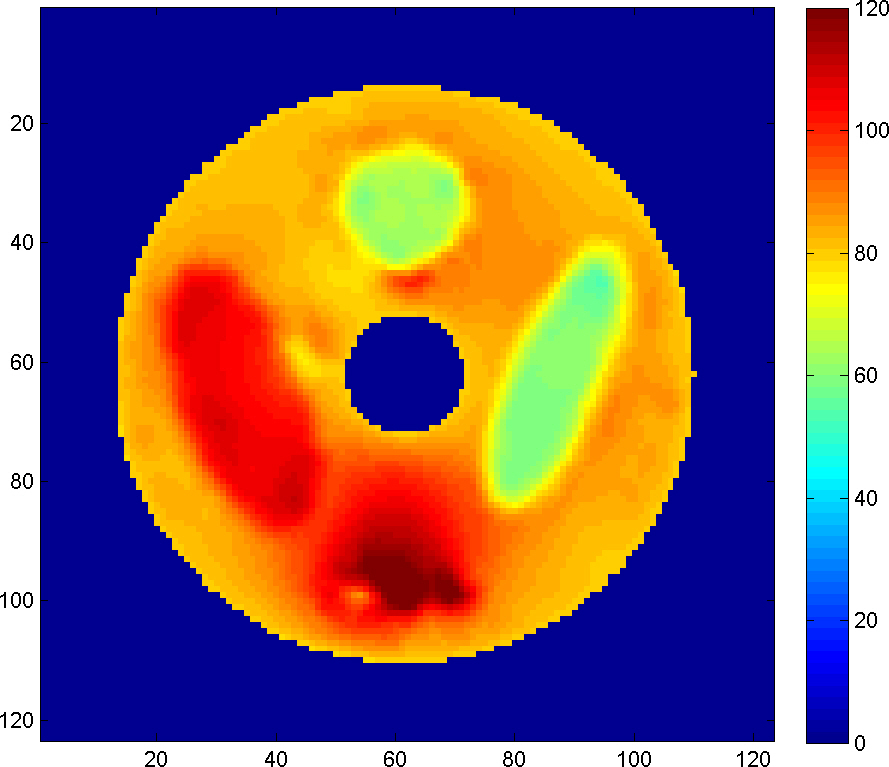}\end{minipage}
    &
    \begin{minipage}{2.4cm}\includegraphics[width=2.4cm]{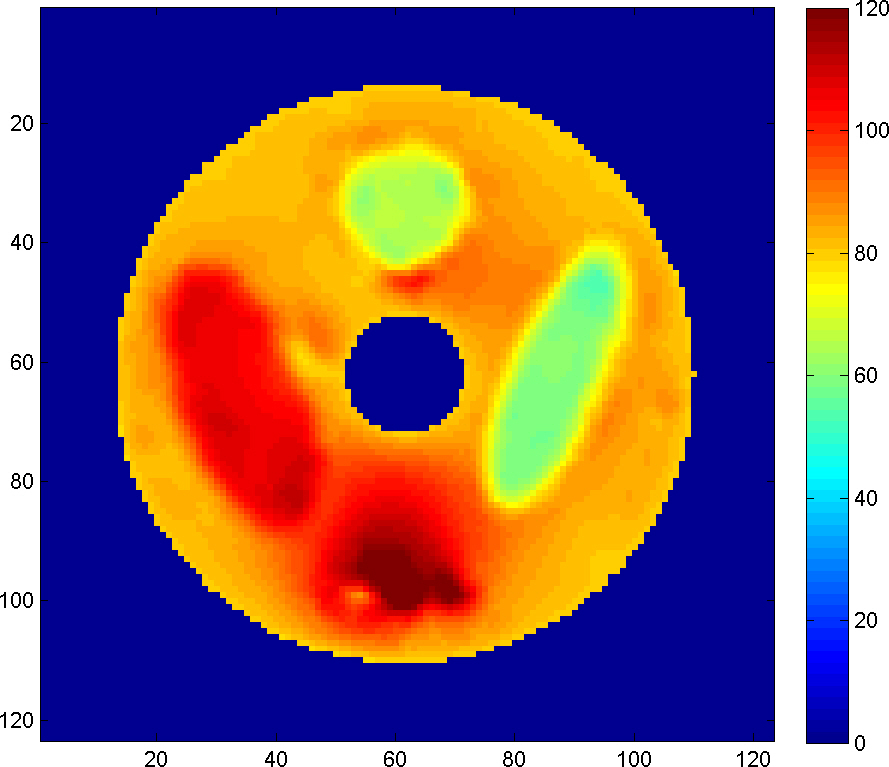}\end{minipage}
    \\
    ~ & (e) $k=1$ & (f) $k=3$ & (g) $k=5$ & (h) $k=10$
  \end{tabular}
  \caption{Reconstruction images obtained from the iterative scheme \eref{eq:newton-poisson-U} with $k=1, 2, 5, 10$ in the slice $\Om_0$. (a)-(d): Images of the first row are reconstructed conductivity distribution. (e)-(h): Images of the second row are reconstructed relative permittivity distribution.}
  \label{fig:sElliptic}
\end{figure}
\begin{figure}
  \begin{tabular}{cc}
    \includegraphics[width=0.45\textwidth]{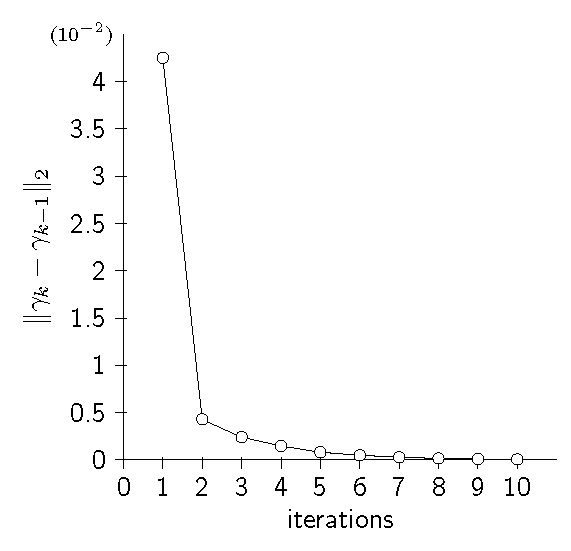}
    &
    \includegraphics[width=0.45\textwidth]{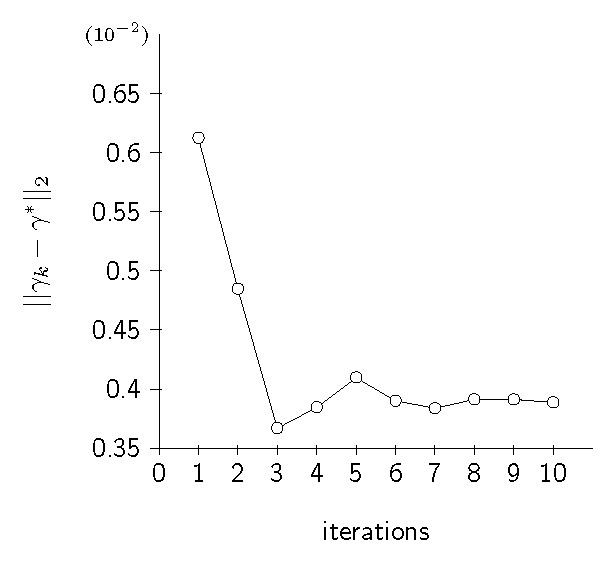}
    \\
    (a) $\|\gamma_k-\gamma_{k-1}\|_2$ & (b) $\|\gamma_k-\gamma^\ast\|_2$
  \end{tabular}
  \caption{(a) Plot $||\gamma_k-\gamma_{k-1}||_2$ to show the convergence of the iteration \eref{eq:newton-poisson-U}. (b) Plot $||\gamma_k-\gamma^\ast||_2$ to show the accuracy of \eref{eq:newton-poisson-U} with iteration numbers $k=1, 2, \cdots, 10$.}
  \label{fig:sElliptic_plot}
\end{figure}
Figure \ref{fig:sElliptic_plot} shows that the iteration method \eref{eq:newton-poisson-U} converges as $k$ increases and the error between true admittivity and reconstructed admittivity decreases. We choose $U_3$ to be the solution of the iterative algorithm.
We used direct method \eref{eq:direct-formula} for the admittivity value $\gamma$ in the segmented subdomain $D$. So, we let $U_3$ with the value obtained from the direct method in $D$ to be the initial guess of the Newton method \eref{eq:newton}.
Figure \ref{fig:newton} illustrates the reconstructed conductivity and relative permittivity distribution by the Newton iteration \eref{eq:newton}. Figure \ref{fig:newton_plot} shows the functional $J[\gamma_n]$ and the accuracy of the Newton method, $\f{1}{|S|}\int_S\left|\f{\gamma_n}{\gamma^*}-1\right|dx$, where $S$ is the region of anomalres. We defined the accuracy criterion to be $\f{1}{|S|}\int_S\left|\f{\gamma_n}{\gamma^*}-1\right|dx$ in order to see the performance of the Newton method only in the regions containing the anomalies.
\begin{figure}
  \begin{tabular}{cccc}
    \includegraphics[width=0.2\textwidth]{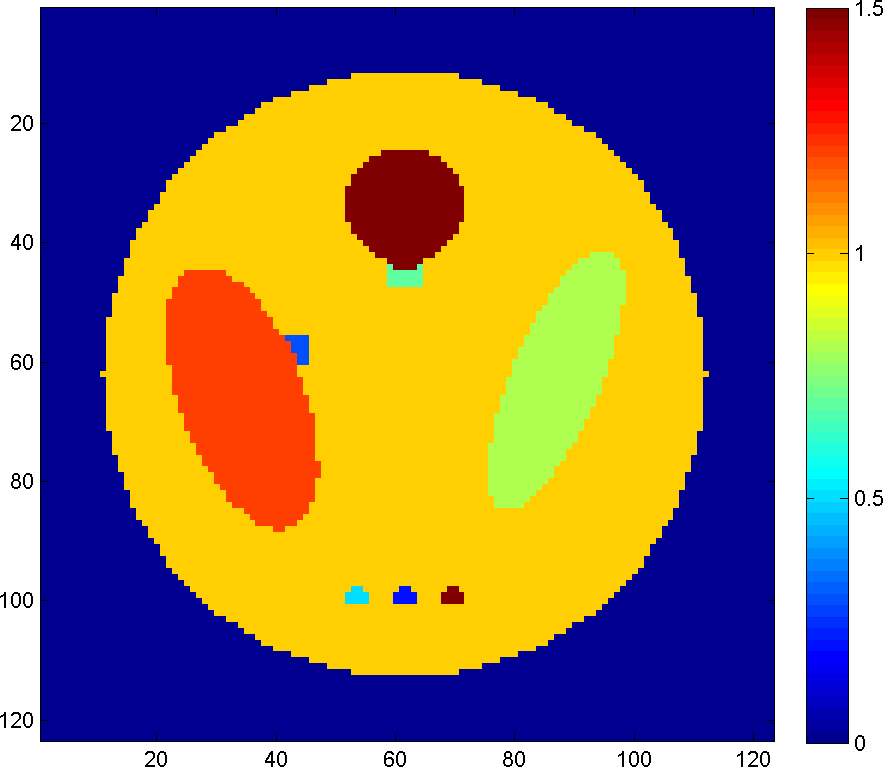}
    &
    \includegraphics[width=0.2\textwidth]{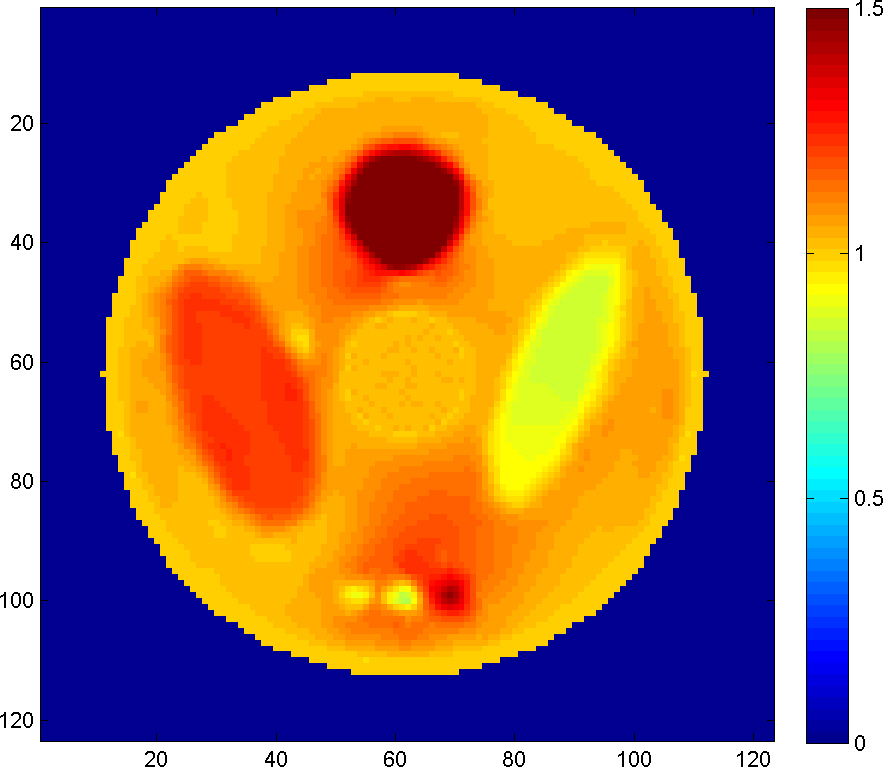}
    &
    \includegraphics[width=0.2\textwidth]{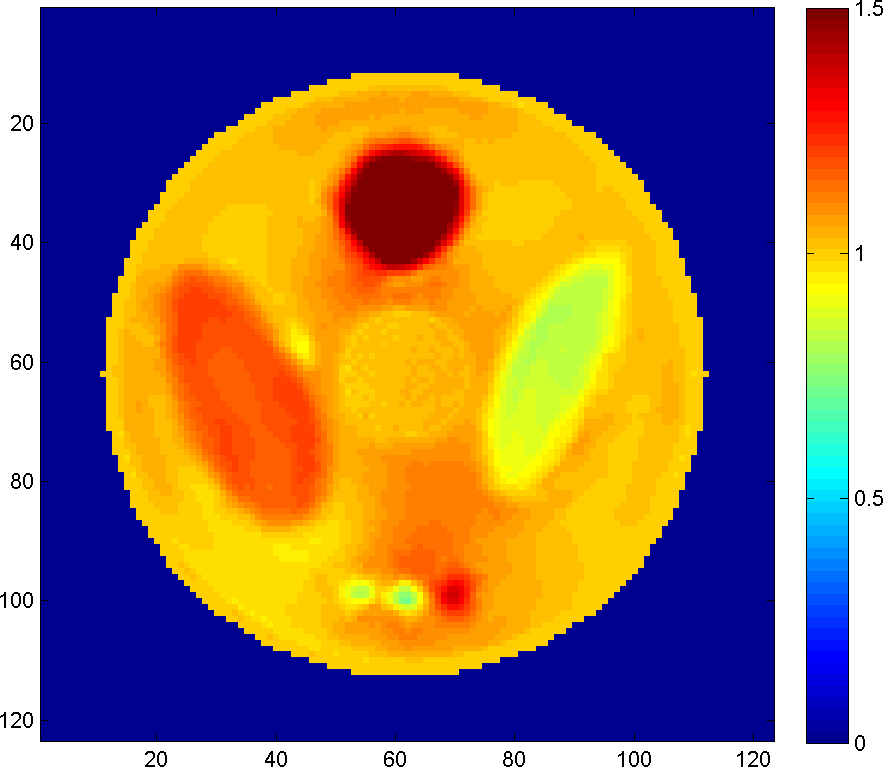}
    &
    \includegraphics[width=0.2\textwidth]{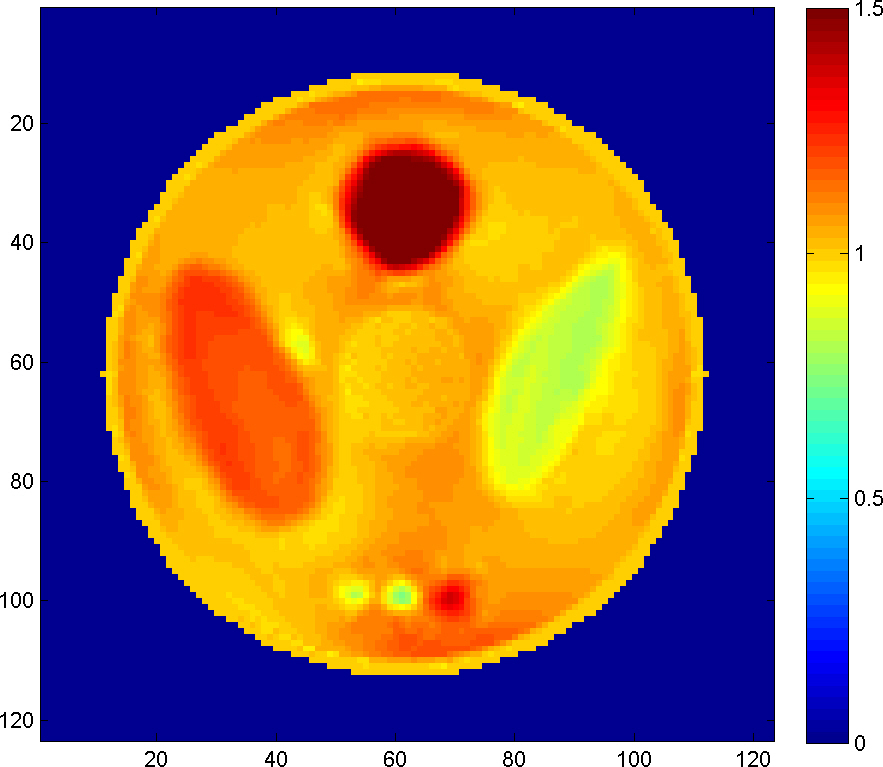}
    \\
    (a) True $\sigma$ & (b) $n=0$ & (c) $n=5$ & (d) $n=10$
    \\
    ~
    \\
    \includegraphics[width=0.2\textwidth]{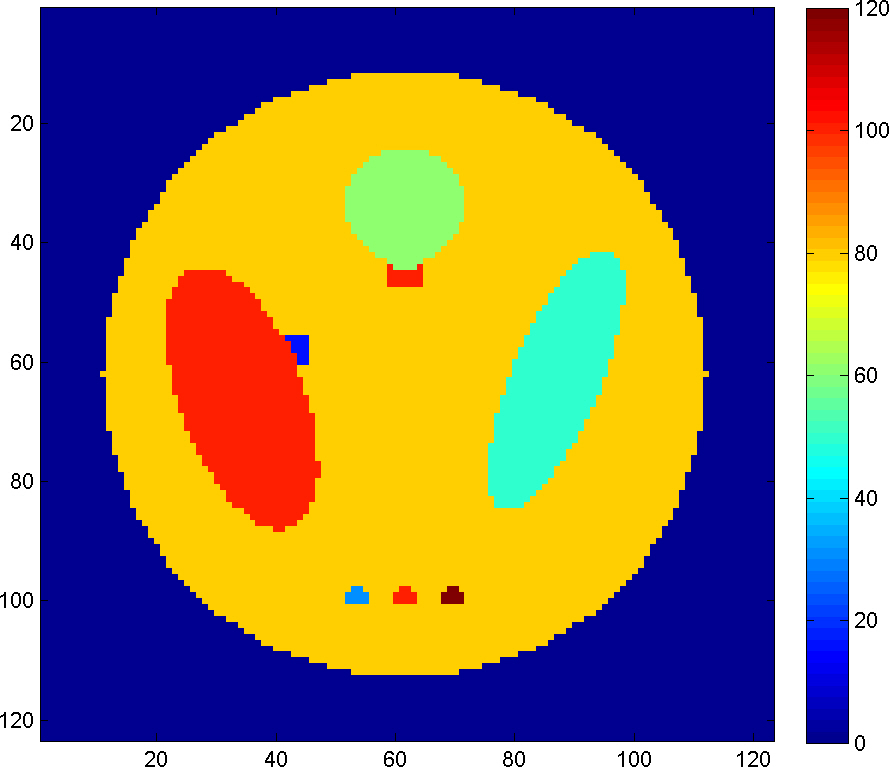}
    &
    \includegraphics[width=0.2\textwidth]{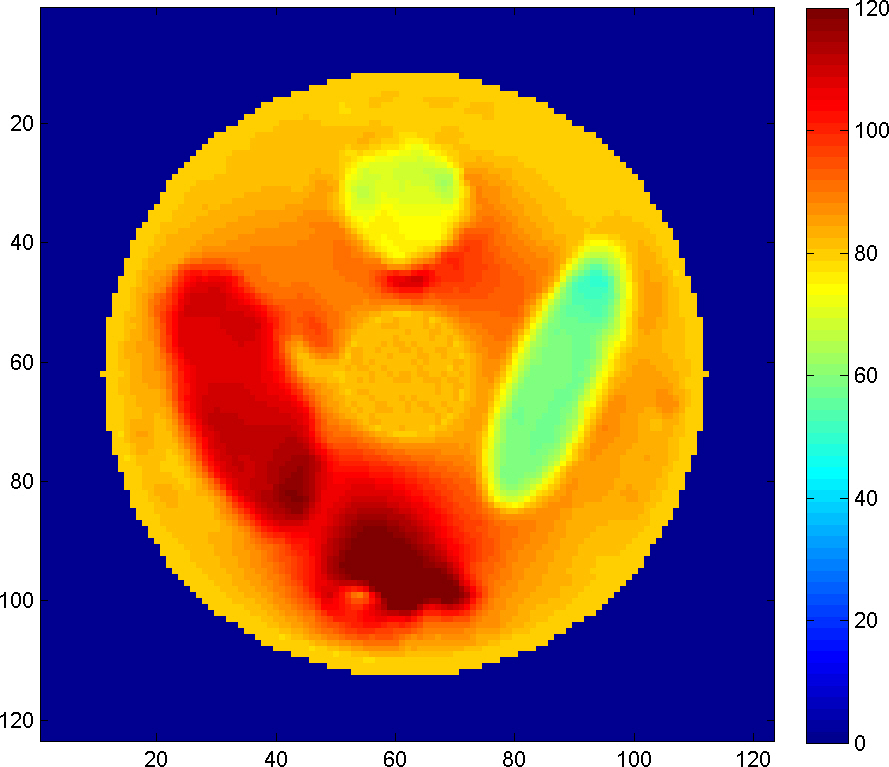}
    &
    \includegraphics[width=0.2\textwidth]{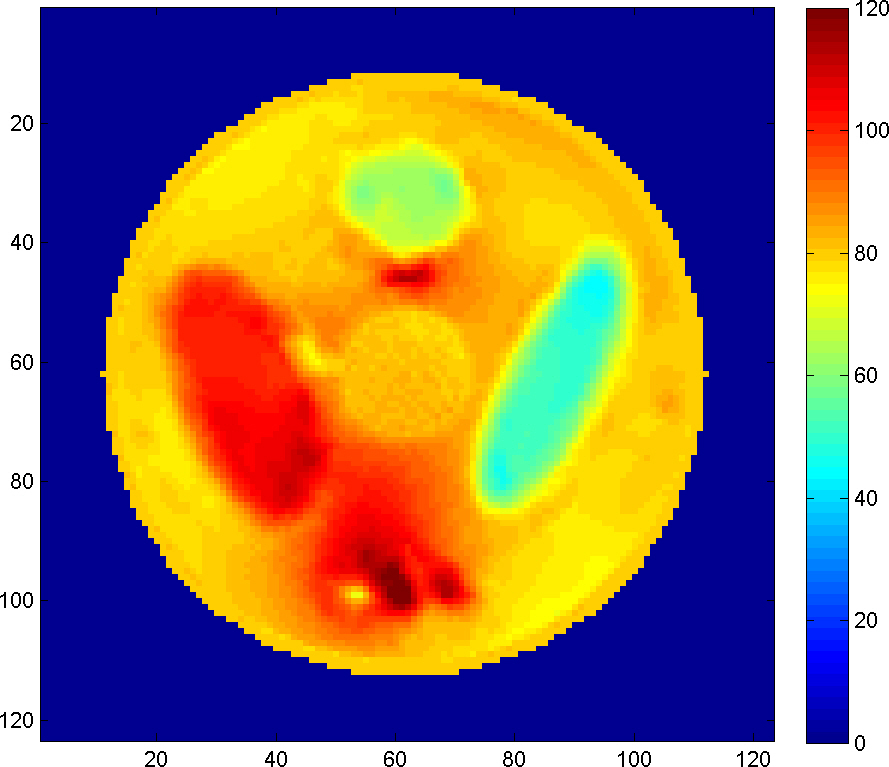}
    &
    \includegraphics[width=0.2\textwidth]{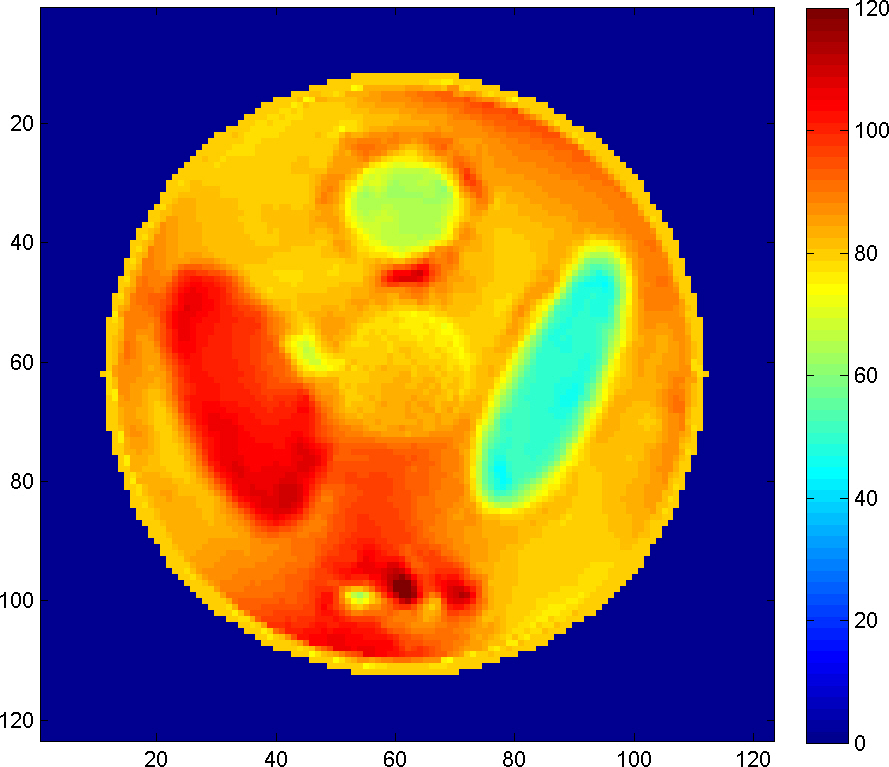}
    \\
    (e) True $\epsilon/\epsilon_0$ & (f) $n=0$ & (g) $n=5$ & (h) $n=10$
  \end{tabular}
  \caption{Reconstruction images obtained from the Newton method \eref{eq:newton} with the iteration numbers $n=1, 2, 5, 10$, in the slice $\Om_0$. (a) and (e) are true conductivity and relative permittivity images, respectively. (b), (c) and (d) are the reconstructed conductivity distributions. (f),(g), (h) are the reconstructed relative permittivity distributions.}
  \label{fig:newton}
\end{figure}
\begin{figure}
  \begin{tabular}{cc}
    \includegraphics[width=0.45\textwidth]{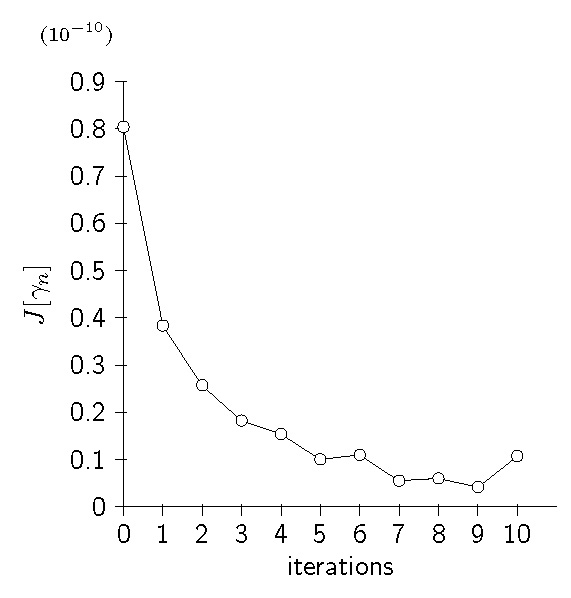}
    &
    \includegraphics[width=0.45\textwidth]{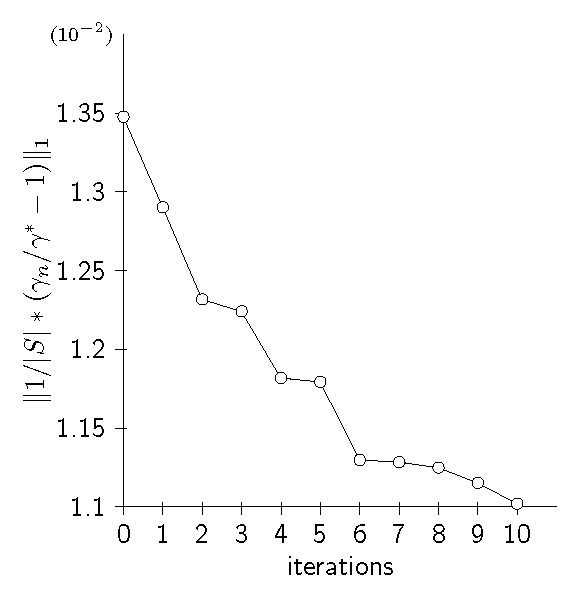}
    \\
    (a) $J[\gamma_n]$ & (b) $\f{1}{|S|}\int_S\left|\f{\gamma_n}{\gamma^*}-1\right|dx$
  \end{tabular}
  \caption{(a) Plot of the functional $J[\gamma_n]$. (b) Plot of $\f{1}{|S|}\int_S\left|\f{\gamma_n}{\gamma^*}-1\right|dx$ to show the accuracy of \eref{eq:newton-poisson-U} with the iteration numbers $n=0, 1,  \ldots, 10$.}
  \label{fig:newton_plot}
\end{figure}

To illustrate the performance of the proposed method, we compared the reconstruction results  with the true values in Figure \ref{fig:compare1}. Figure \ref{fig:compare2} shows the reconstructed images by the direct formula \eref{eq:direct-formula} and the proposed method. Figure \ref{fig:compare3} compares between the two methods, the direct formula \eref{eq:direct-formula} and the proposed method, for imaging small anomalies.
\begin{figure}
  \begin{tabular}{ccc}
    \includegraphics[width=3.8cm]{true_con}
    &
    \includegraphics[width=3.8cm]{sElliptic_con_it3}
    &
    \includegraphics[width=3.8cm]{Newton_con_it10}
    \\
    (a) & (b) & (c)
    \\
    ~
    \\
    ~
    &
    \includegraphics[width=3.8cm]{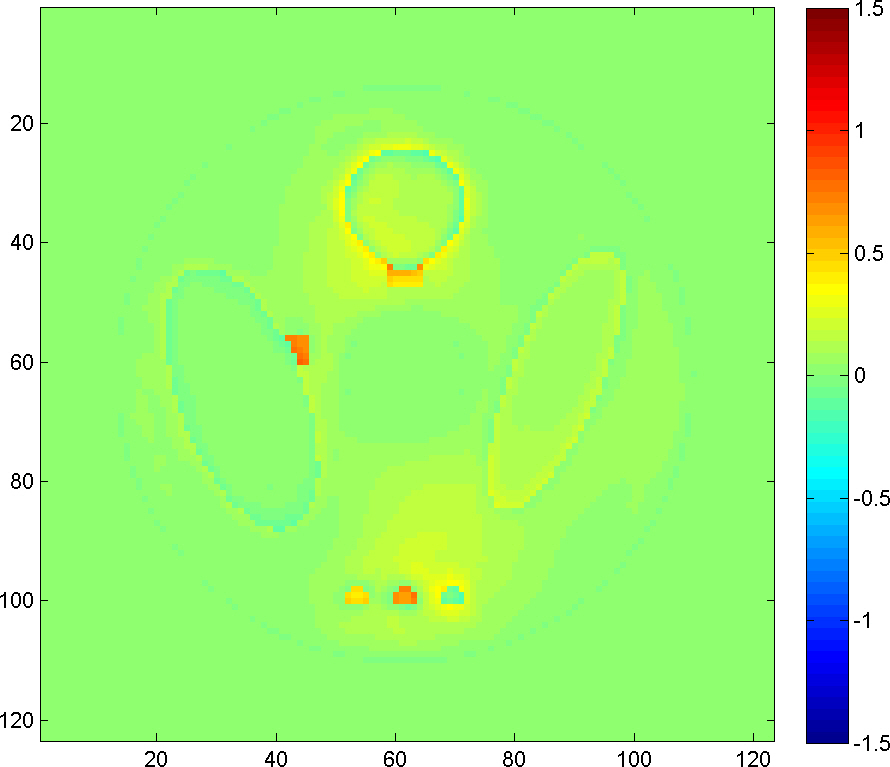}
    &
    \includegraphics[width=3.8cm]{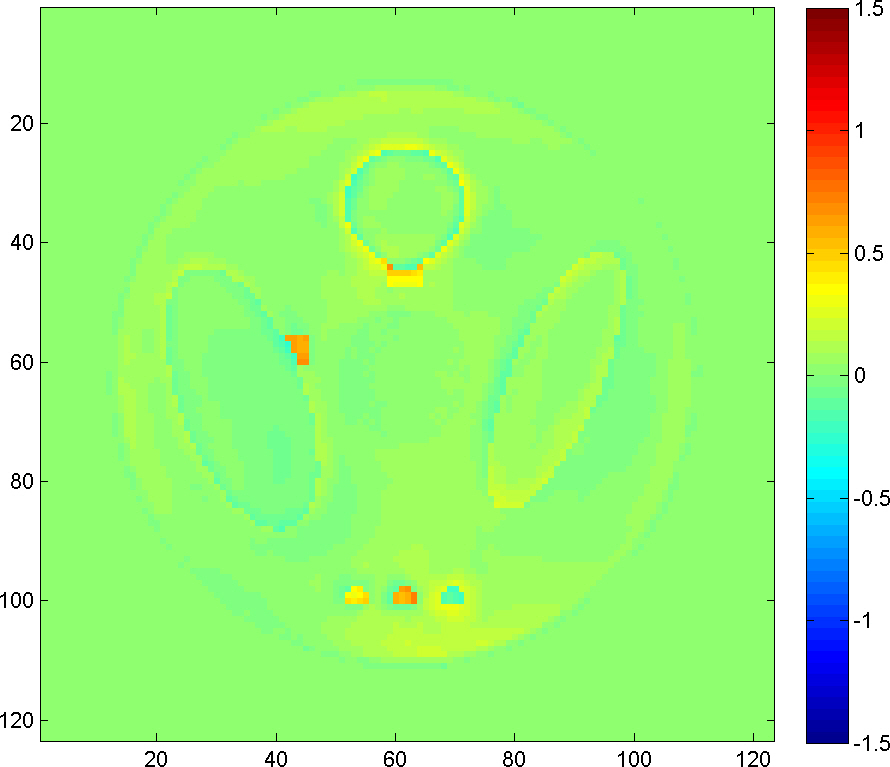}
    \\
    ~ & (d) & (e)
  \end{tabular}
  \caption{(a) True conductivity distribution in the slice $\Om_0$. (b) Reconstructed conductivity image by the semi-elliptic PDE \eref{eq:newton-poisson-U} with $k=3$ in the segmented slice $\Om_0\backslash D$. (c) Reconstructed conductivity image by the Newton iteration method \eref{eq:newton} with $n=10$ in the slice $\Om_0$. (d) The image of the error of (b) in $\Om_0\backslash D$. (e) The image of the error of (c) in $\Om_0$.}
  \label{fig:compare1}
\end{figure}
\begin{figure}
  \begin{tabular}{ccc}
    \includegraphics[width=3.8cm]{true_con}
    &
    \includegraphics[width=3.8cm]{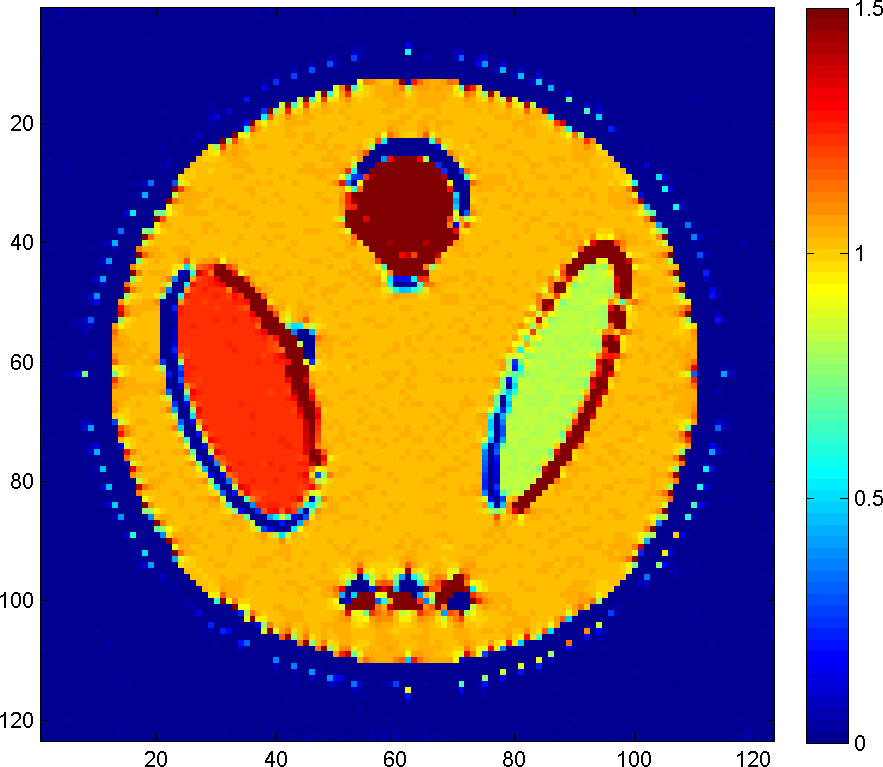}
    &
    \includegraphics[width=3.8cm]{Newton_con_it10}
    \\
    (a) & (b) & (c)
  \end{tabular}
  \caption{(a) True conductivity in the slice $\Om_0$. (b) Reconstructed conductivity obtained from the direct formula \eref{eq:direct-formula} in $\Om_0$. (c) Reconstructed conductivity using the proposed method in $\Om_0$.}
  \label{fig:compare2}
\end{figure}
\begin{figure}
  \includegraphics[width=12cm]{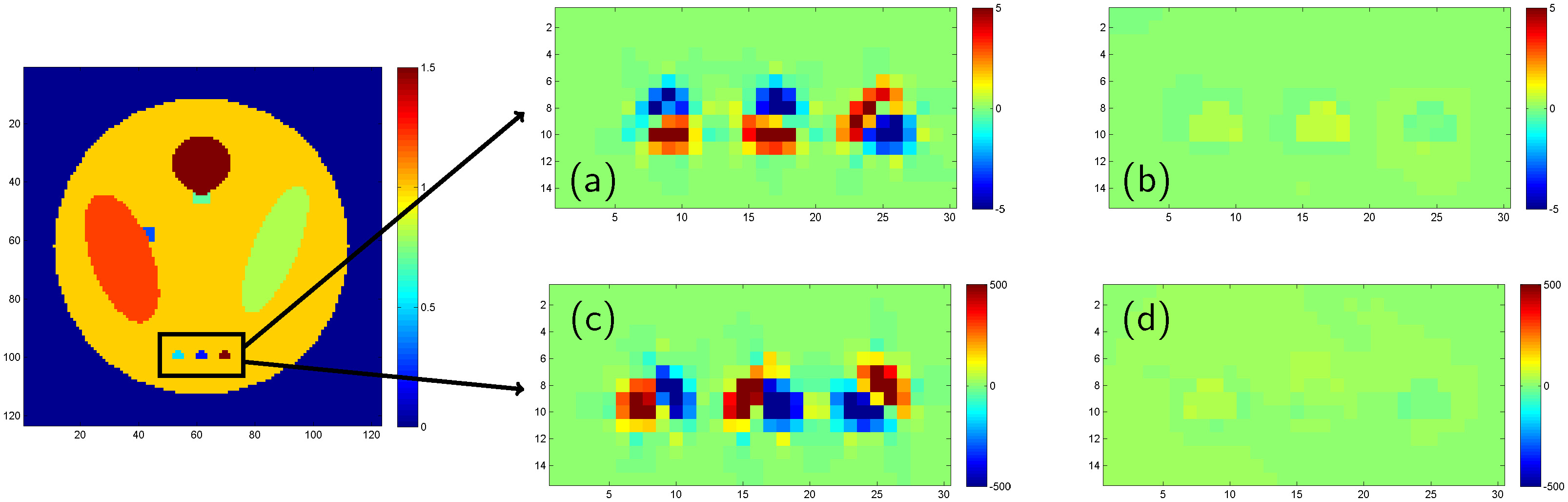}
  \caption{(a) and (c), respectively, are images of the errors of conductivity and relative permittivity using the direct formula \eref{eq:direct-formula} in the slice $\Om_0$. (b) and (d), respectively, are images of the errors of conductivity and relative permittivity using the proposed method in the slice $\Om_0$.}
  \label{fig:compare3}
\end{figure}

In the second numerical model, we simulate the model with admittivity changing along $z$-direction, i.e., $\f{\p \gamma}{\p z} \ne 0$. The domain $\Om$ is decomposed into two parts, $\Om_-=\Om\cap\{z<0\}$ and $\Om_+=\Om\cap\{z\geq0\}$. In $\Om_-$, the admittivity distribution is the same as in Model 1. However, the admittivity distribution in $\Om_+$ is different from Model 1 and is such that $\f{\p \gamma}{\p z} \ne 0$ in $\Om_0$. Figure \ref{fig:model2} shows the second configuration model, the conductivity and the relative permittivity values in the domain. Figure \ref{fig:result_model2} shows the reconstruction results  using the direct method and the proposed method. Figure \ref{fig:Err_plot_model2} shows the accuracy of the proposed method applied to the second model. Figure \ref{fig:compare1_model2} presents the reconstructed conductivity distribution of the second model in the slice of $\Om_-$ using the proposed method. Figure \ref{fig:Err_plot_model2} and Figure \ref{fig:compare1_model2} demonstrate that the proposed method works well in the case of $\f{\p \gamma}{\p z} \ne 0$.
\begin{figure}
  \includegraphics[width=1.0\textwidth]{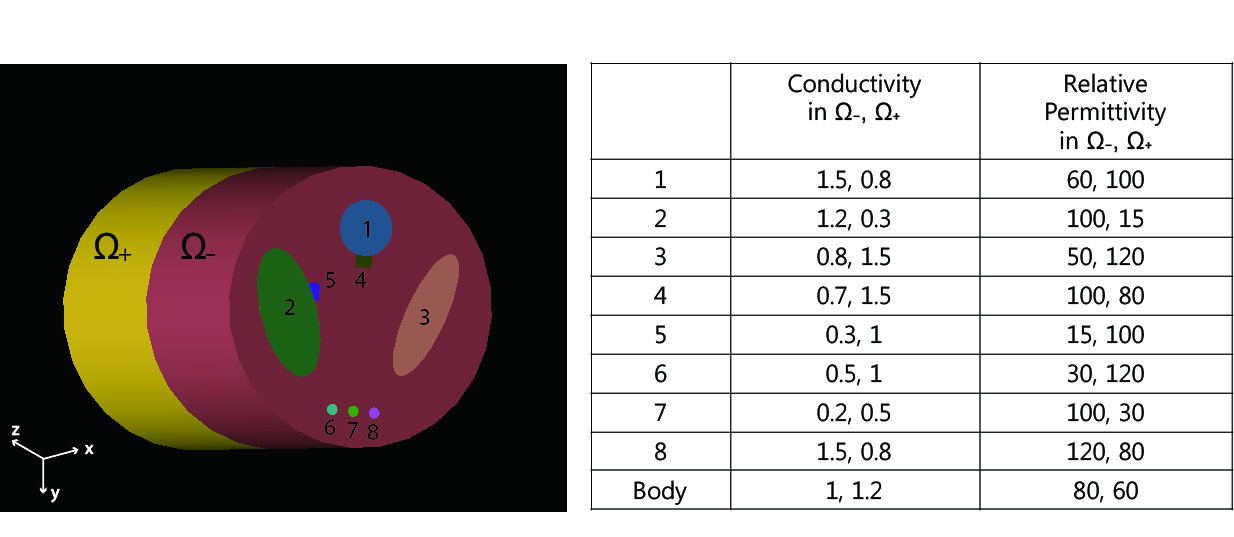}
  \caption{Second model configuration (left) and table of the value of electrical property (right).}
  \label{fig:model2}
\end{figure}
\begin{figure}
  \begin{tabular}{ccc}
    \includegraphics[width=3.8cm]{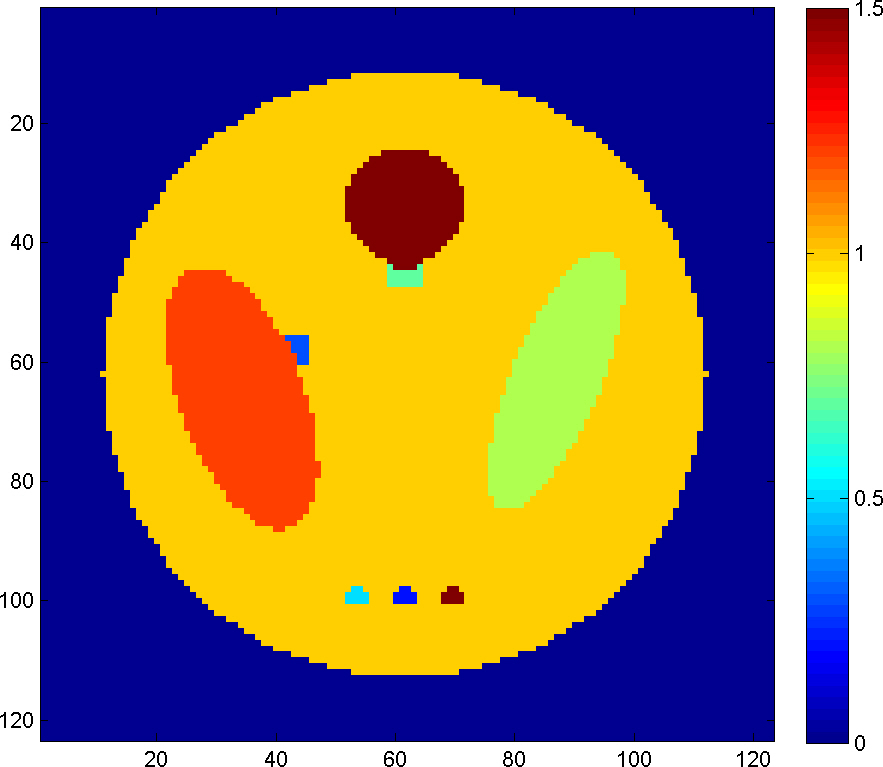} &
    \includegraphics[width=3.8cm]{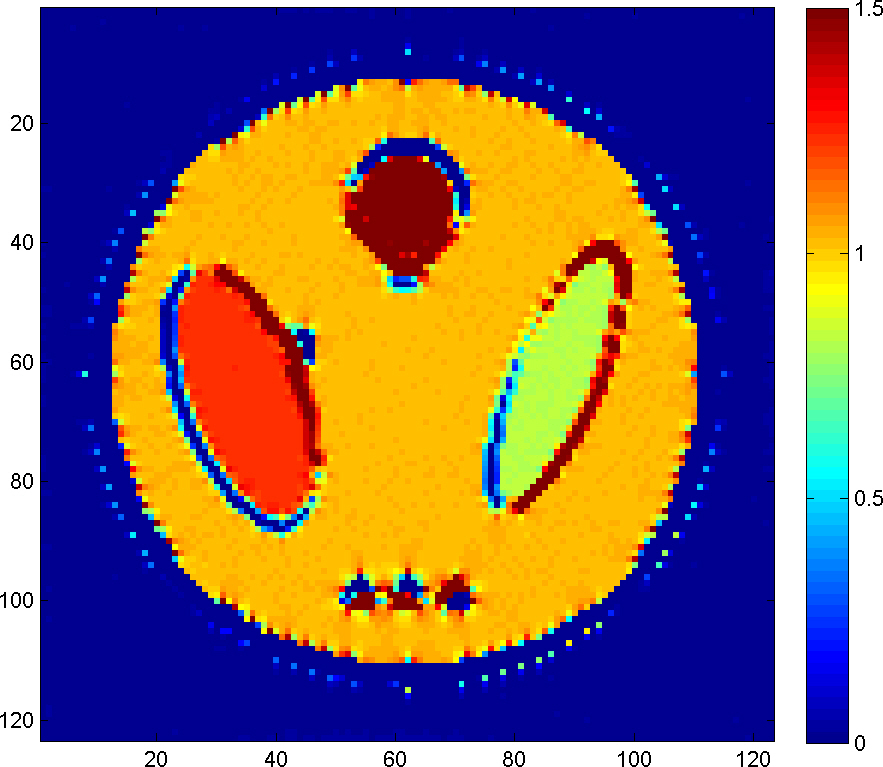} &
    \includegraphics[width=3.8cm]{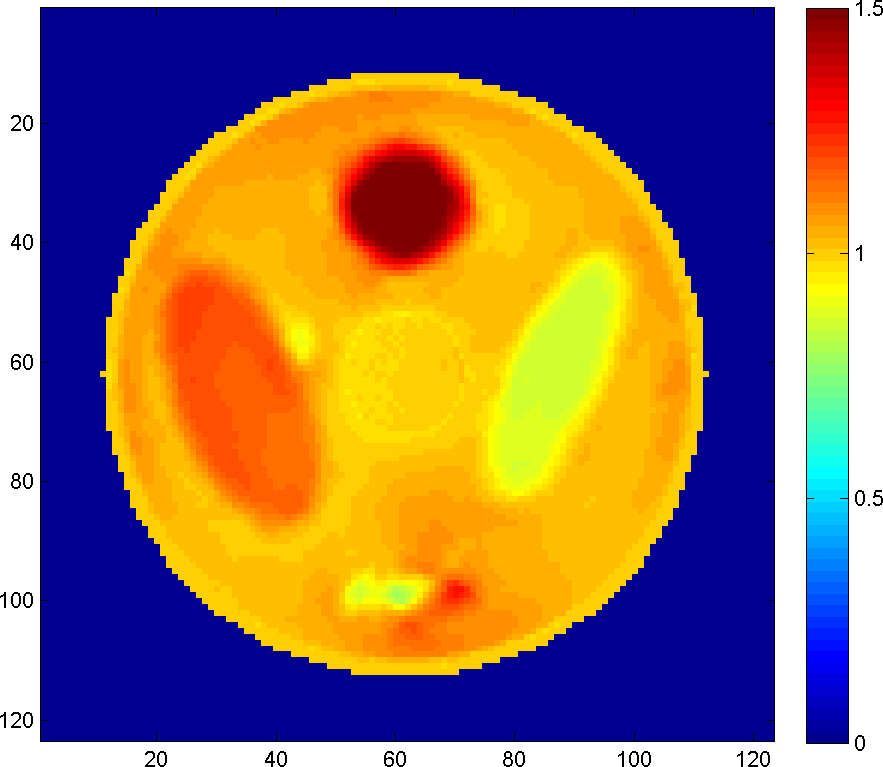} \\
    (a) & (b) & (c) \\
    \includegraphics[width=3.8cm]{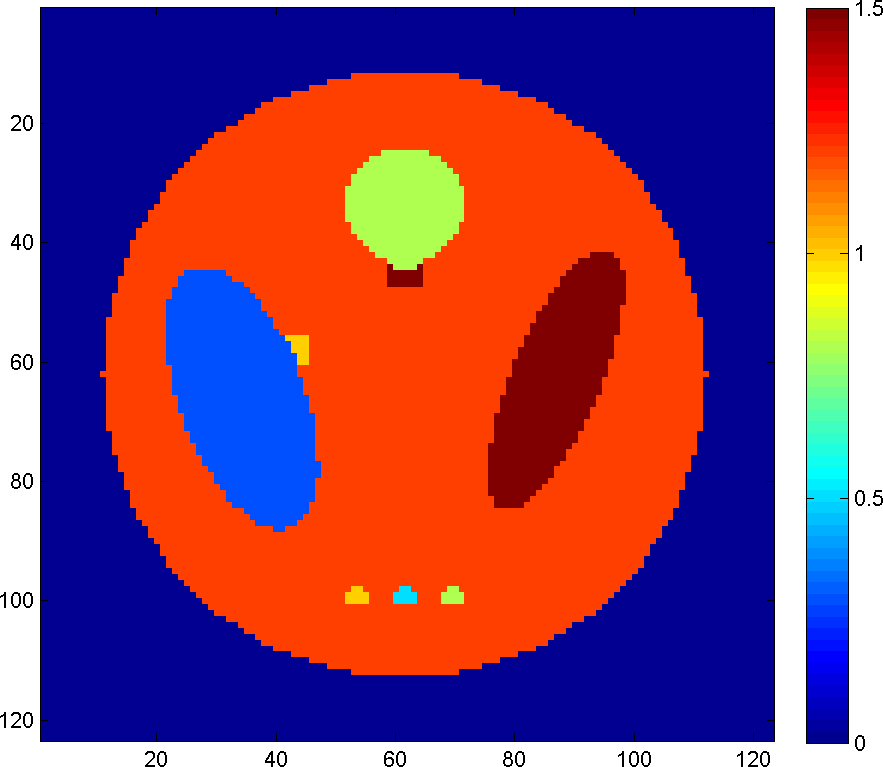} &
    \includegraphics[width=3.8cm]{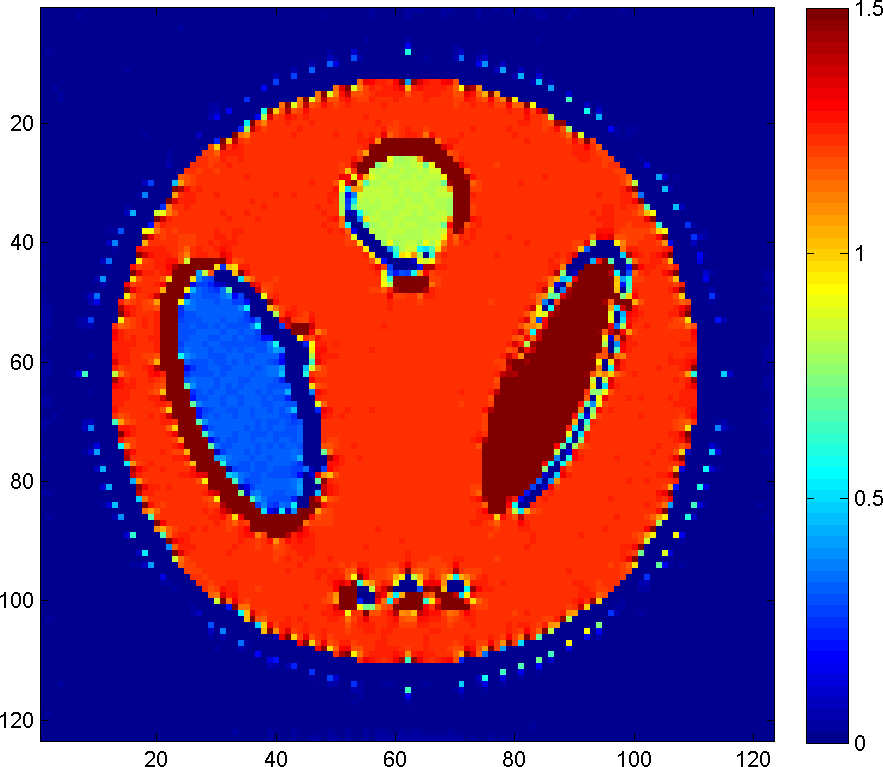} &
    \includegraphics[width=3.8cm]{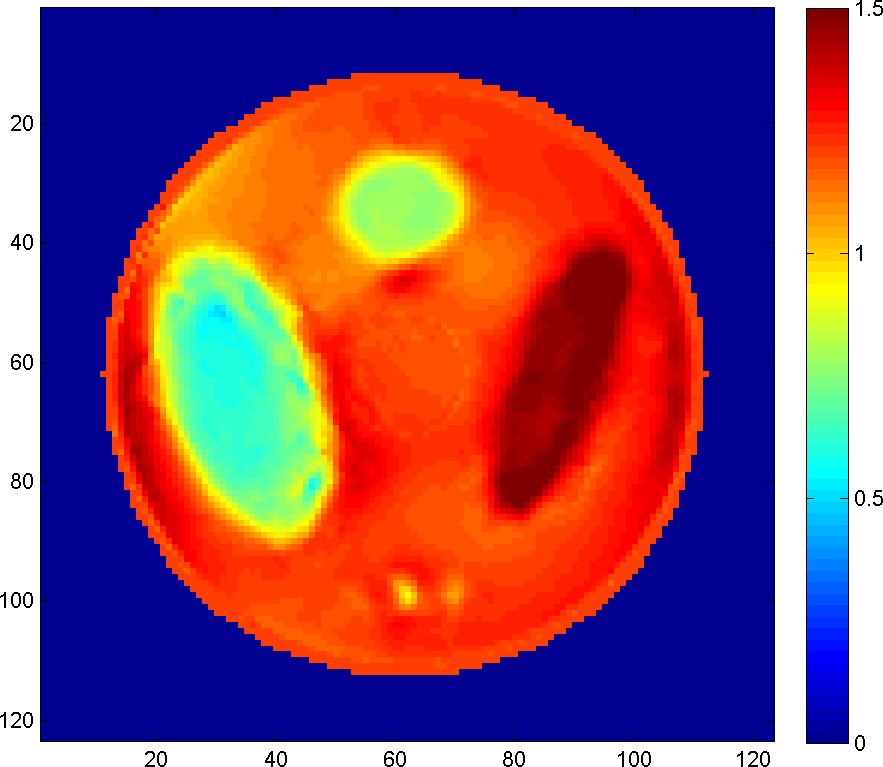} \\
    (d) & (e) & (f)
  \end{tabular}
  \caption{(a) True conductivity in the slice in $\Om_-$. (b) Reconstructed conductivity obtained from the direct formula \eref{eq:direct-formula} in $\Om_-$. (c) Reconstructed conductivity using the proposed method in $\Om_-$. (d) True conductivity in the slice in $\Om_+$. (e) Reconstructed conductivity obtained from the direct formula \eref{eq:direct-formula} in $\Om_+$. (f) Reconstructed conductivity using the proposed method in $\Om_+$.}
  \label{fig:result_model2}
\end{figure}
\begin{figure}
  \begin{tabular}{cc}
    \includegraphics[width=0.45\textwidth]{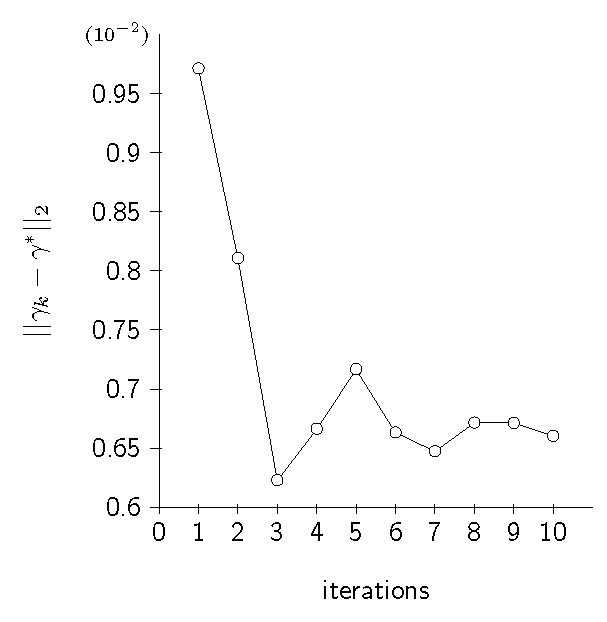}
    &
    \includegraphics[width=0.45\textwidth]{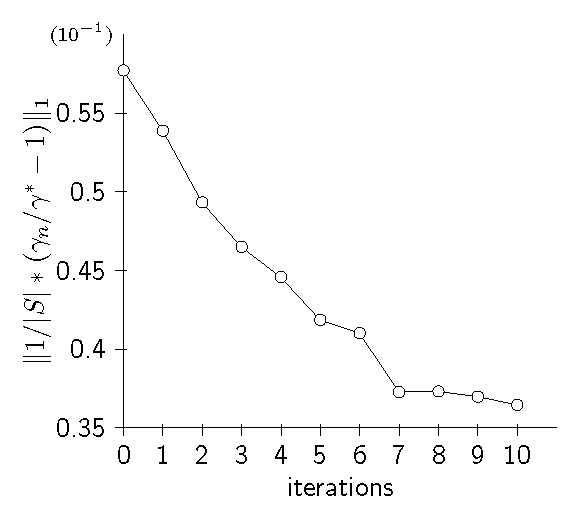}
    \\
    (a) $\|\gamma_k-\gamma^\ast\|_2$ & (b) $\f{1}{|S|}\int_S\left|\f{\gamma_n}{\gamma^*}-1\right|dx$
  \end{tabular}
  \caption{(a) Plot of $||\gamma_k-\gamma^\ast||_2$ to show the accuracy of the semi-elliptic PDE \eref{eq:newton-poisson-U} with the iteration numbers $k=1, 2, \cdots, 10$. (b) Plot of $\f{1}{|S|}\int_S\left|\f{\gamma_n}{\gamma^*}-1\right|dx$ to show the accuracy of the Newton iteration \eref{eq:newton-poisson-U} with the iteration numbers $n=0, 1,  \cdots, 10$.}
  \label{fig:Err_plot_model2}
\end{figure}
\begin{figure}
  \begin{tabular}{ccc}
    \includegraphics[width=3.8cm]{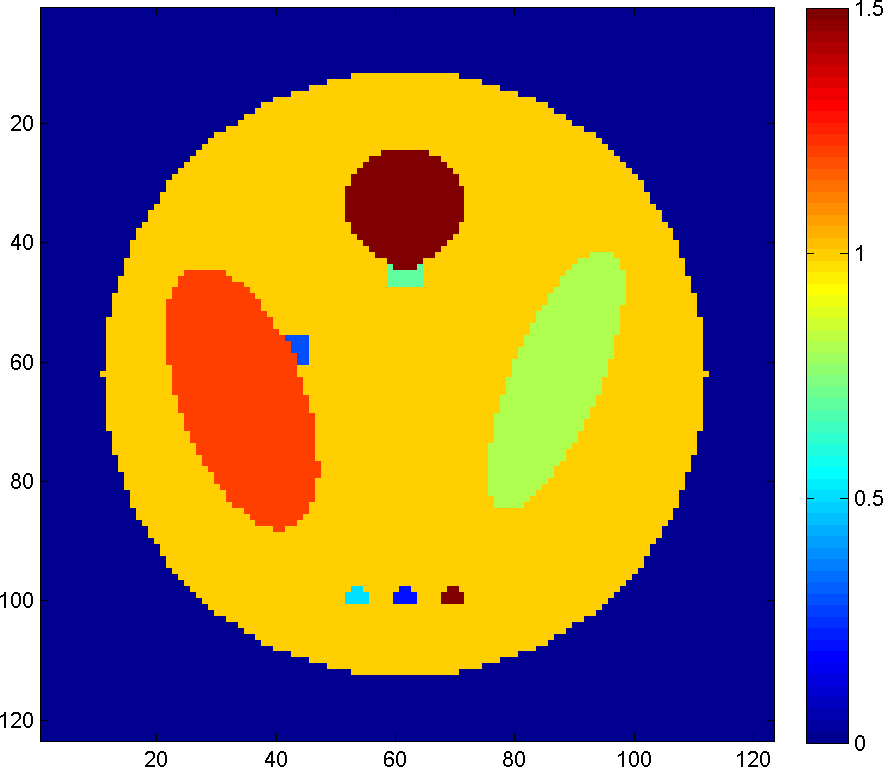}
    &
    \includegraphics[width=3.8cm]{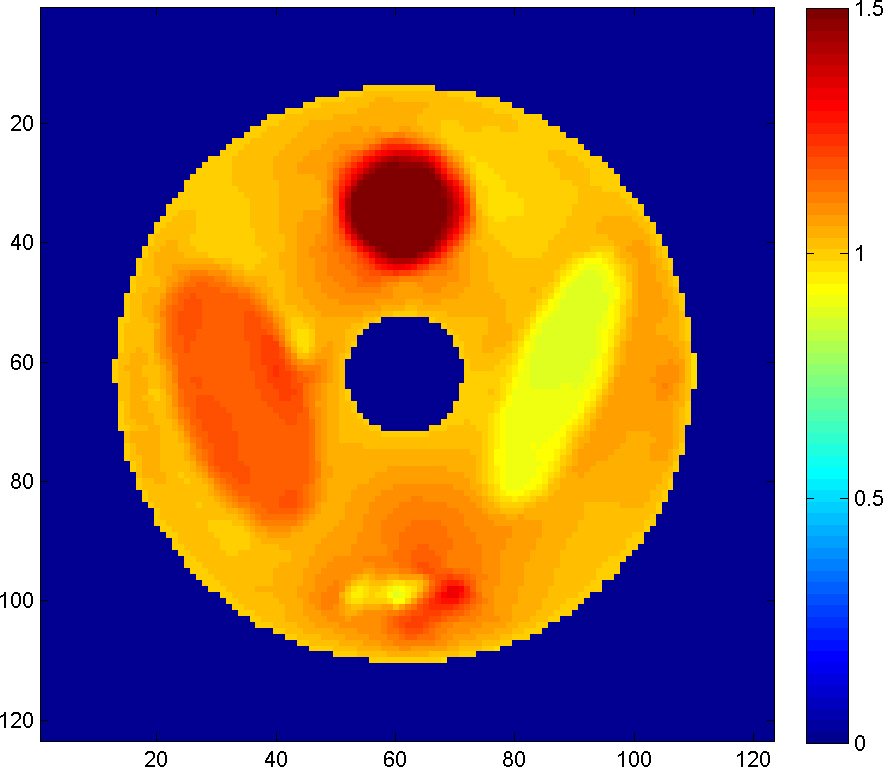}
    &
    \includegraphics[width=3.8cm]{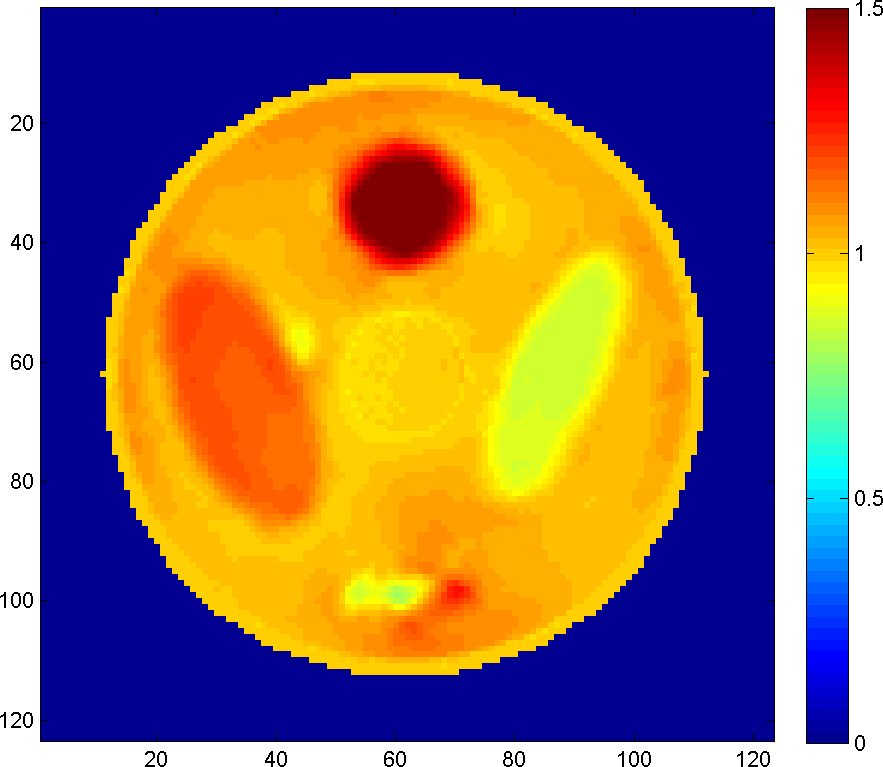}
    \\
    (a) & (b) & (c)
    \\
    ~
    \\
    ~
    &
    \includegraphics[width=3.8cm]{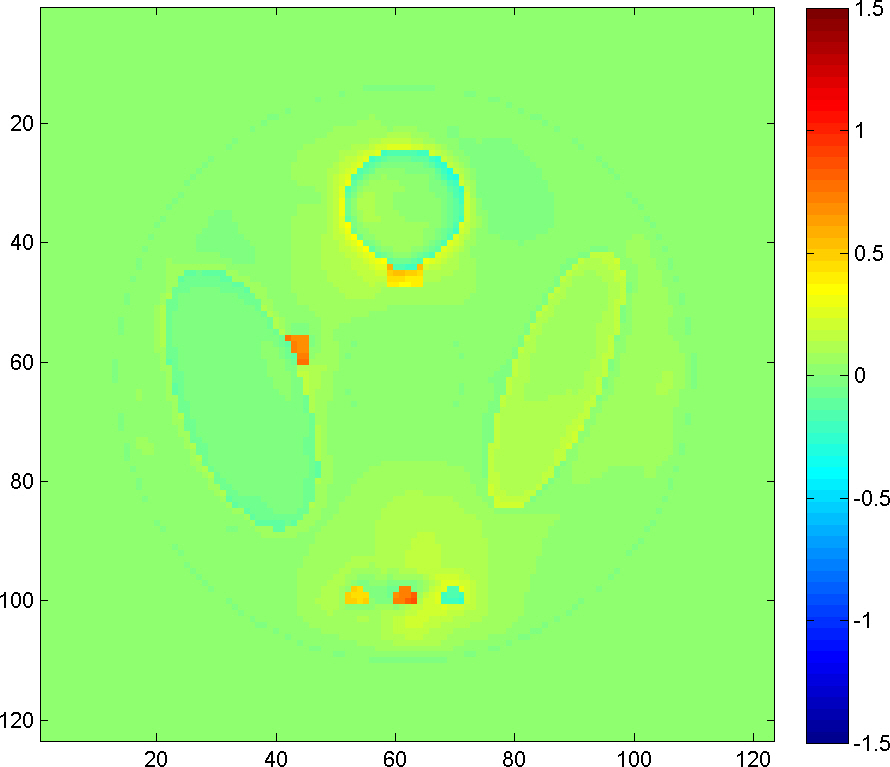}
    &
    \includegraphics[width=3.8cm]{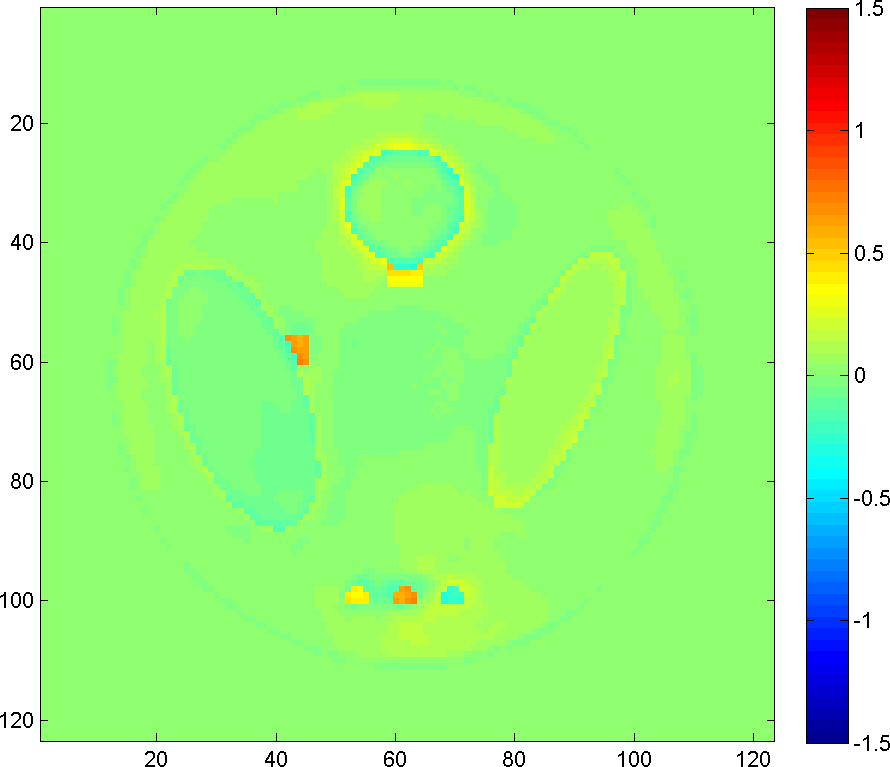}
    \\
    ~ & (d) & (e)
  \end{tabular}
  \caption{(a) True conductivity distribution in the slice of $\Om_-$. (b) Reconstructed conductivity image by the semi-elliptic PDE \eref{eq:newton-poisson-U} with $k=3$ in the segmented in the slice of $\Om_-\backslash D$. (c) Reconstructed conductivity image by the Newton iteration method \eref{eq:newton} with $n=10$ in the slice of $\Om_-$. (d) The image of the error of (b) in the slice of $\Om_-\backslash D$. (e) The image of the error of (c) in the slice of $\Om_-$.}
  \label{fig:compare1_model2}
  \end{figure}

\section{Concluding remarks}

In this paper, we have developed  an iterative novel scheme for reconstructing electrical tissue properties at the Larmor frequency from measurements of the positive rotating magnetic field. We first suggest the elliptic partial differential equation \eref{eq:poisson-U} which provides a blurred reconstructed image. By considering the blurred reconstructed image as an initial guess of the Newton iteration, the Newton iteration for finding the minimizer of the functional $J$ in \eref{minproblem} finds the final reconstruction admittivity. Note that our scheme does not require a local homogeneity assumptions on $\gamma$ and allows to reconstruct inhomogeneous distributions accurately.





\end{document}